\documentclass[a4paper]{article}

\usepackage[utf8]{inputenc}
\usepackage[T1]{fontenc}

\usepackage{amsmath}
\usepackage{amssymb}
\usepackage{hyperref}

\newcommand{\C}{\mathbb{C}}
\newcommand{\R}{\mathbb{R}}
\newcommand{\N}{\mathbb{N}}

\newcommand{\Hi}{\mathcal{H}}

\newcommand{\Gi}{\mathcal{G}}

\newcommand{\Vc}{\mathcal{V}}
\newcommand{\Lc}{\mathcal{L}}
\newcommand{\dom}{{\rm dom\,}}
\newcommand{\sgn}{{\rm sgn\,}}
\newcommand{\supp}{{\rm supp\,}}

\newcommand{\op}{{\rm op\,}}

\newcommand{\ran}{{\rm ran\,}}

\newcommand{\Span}{{\rm span\,}}

\newcommand{\mul}{{\rm mul\,}}

\newcommand{\loc}{{\rm loc\,}}

\newcommand{\Eplus}{\text{\footnotesize $\bigoplus\limits_{e\in E} $}\,}

\newcommand{\plus}{\text{\footnotesize $\bigoplus\limits_{n\in\N} $}\,}

\newcommand{\nplus}{\text{\footnotesize $\bigoplus\limits_{n=0}^{\infty} $}\,}

\usepackage[amsmath,amsthm,thmmarks]{ntheorem}

\theoremstyle{plain}
\newtheorem{Definition}{Definition}[section]
\newtheorem{Theorem}[Definition]{Theorem}
\newtheorem{Proposition}[Definition]{Proposition}
\newtheorem{Lemma}[Definition]{Lemma}
\newtheorem{Corollary}[Definition]{Corollary}
\newtheorem{Example}[Definition]{Example}

\title{Locally finite extensions and Gesztesy-\v{S}eba
realizations for the Dirac operator on a metric graph
        }

\author{Hannes Gernandt\thanks{Institute of Mathematics, TU Ilmenau, Weimarer Stra\ss e 25, 98693 Ilmenau, Germany ({\tt hannes.gernandt@tu-ilmenau.de}).}
        \and Carsten Trunk\thanks{Institut f\"{u}r Mathematik,  TU Ilmenau, Postfach 100565, D-98694, Ilmenau, Germany and Instituto Argentino de Mathem\'{a}tica "Alberto P. Calder\'{o}n" (CONICET), Saavedra 15, (1083) Buenos Aires, Argentina ({\tt carsten.trunk@tu-ilmenau.de}). }}

\begin{document}

\maketitle

\begin{abstract}
 We study extensions of direct sums of symmetric operators $S=\oplus_{n\in\N} S_n$.
  In general there is no natural boundary triplet for $S^*$ even if there is one for every $S_n^*$,
   $n\in\N$. We consider a subclass of extensions of $S$ which can be described in terms
    of the boundary triplets of $S_n^*$ and investigate the self-adjointness, the semi-boundedness
     from below and the discreteness of the spectrum. Sufficient conditions for these properties
      are obtained from recent results on weighted discrete Laplacians.
 The results are applied to Dirac operators on metric graphs with point interactions at the vertices.
  In particular, we allow graphs with arbitrarily small edge length.
\end{abstract}







\section{Introduction}
We consider direct sum operators $S=\plus S_n$ in a direct sum Hilbert space $\Hi=\plus \Hi_n$ associated to a family of closed densely defined symmetric operators $\{S_n\}_{n\in\N}$, where  $S_n$ is defined in the Hilbert space $\Hi_n$. It is easy to see that  $S$ is closed and symmetric. Furthermore, if $S_n$ has self-adjoint extensions for all $n\in\N$, then also $S$ has self-adjoint extensions.

The direct sum operator $S$ can be viewed as an diagonal operator matrix with infinitely many entries. Its
self-adjoint extensions are no longer diagonal. Here we are interested in the spectrum and related properties. Setting
$\Hi_n=\{0\}$ for all but for two or three entries we end up with a $2 \times 2$ ($3\times 3$, respectively) operator matrix,
see the books \cite{Tretter2008} and \cite{Jeribi}.

For the description of the extensions of closed symmetric operators and their spectral properties we use boundary triplets and their associated Weyl functions, see \cite{DM91,DM95,GG91,K75}. A boundary triplet $\{\Gi,\Gamma_0,\Gamma_1\}$ consists of a Hilbert space $\Gi$ and a surjection $(\Gamma_0,\Gamma_1)^T:\dom S^*\rightarrow \Gi\times \Gi$ that satisfies an abstract Green identity, cf.\ \eqref{GI} below. Here the closed extensions of $S$ correspond one to one to the closed linear subspaces $\Theta\subseteq\Gi\times \Gi$ and the extension of $S$ is given by
\begin{align}
\label{Stheta}
S_{\Theta}:=\{f\in\dom S^* ~|~ (\Gamma_0 f,\Gamma_1 f)\in\Theta\}.
\end{align}

In order to apply this approach to quantum graphs, we will write the extension \eqref{Stheta} of $S$ in a different, more suitable way: given a closed subspace $\Gi_{\op}$ of $\Gi$ and a closed operator $L$ with $\dom L\subseteq\Gi_{\op}$ then a specific closed extension of $S$ is given by
\begin{align}
\label{SPLint}
S_{L}=\{f\in\dom S^*~|~L\Gamma_0f=P_{\Gi_{\op}}\Gamma_1 f,\ \Gamma_0f\in\Gi_{\op}\}.
\end{align}

To illustrate the above abstract concept, we will briefly show how \eqref{SPLint} looks like for a $\delta$-type point interaction on a graph $G$ with countable sets of vertices $V$ and edges $E$ and with the edge length function $\ell:E\rightarrow(0,\infty)$. Consider  $\Hi=L^2(G)=\Eplus L^2(0,\ell(e))$ with the operator
\begin{align}
\label{Sbsp}
S=\Eplus S_{e},\quad  \dom S_{e}=W^{2,2}_0(0,\ell(e)),\quad S_{e}:=-\frac{d^2}{dx_e^2},
\end{align}
where $S_{e}$ is the minimal operator on the edge $e$ associated with the differential expression above and $W^{2,2}_0(0,\ell(e))$ denotes the usual second order Sobolev space with boundary values equal to zero. The operator $S$ in \eqref{Sbsp} is symmetric with the adjoint $S^*$ defined on $W^{2,2}(G):=\Eplus W^{2,2}(0,\ell(e))$.

A point interaction of $\delta$-type on a graph is an extension $H_{\alpha}$ of $S$. It is introduced for finite graphs in \cite{BK13,BEH08} 
and for infinite graphs in \cite{EKMN17}. The domain of $H_{\alpha}$ can be specified with a real-valued sequence $(\alpha(v))_{v\in V}$ by 
\begin{align}
\label{PIintro}
\dom H_{\alpha}:=\Big\{(\psi_e)_{e\in E}\in W^{2,2}(G)\cap \mathcal{C}(G)~\Big|~\sum_{(e,t)\in I_v}\sgn(e,t)\psi_e'(t\ell(e))=\alpha(v)\psi(v),~v\in V \Big\},
\end{align}
where $\mathcal{C}(G)$ is the set of continuous functions on $G$ viewed as a metric space, $\psi(v)$ is the evaluation of $\psi$ at the vertex $v$ and $I_v$ is the set of pairs $(e,t)$ with $e\in E$, $t=0,1$. We have $(e,0)\in I_v$ if $v$ is an initial vertex of the directed edge $e$ and in this case we set $\sgn(e,0):=1$. Furthermore, we have $(e,1)\in I_v$ if $v$ is a terminal vertex of the directed edge $e$ and we set $\sgn(e,1):=-1$.

We show how \eqref{PIintro} can be written in the form \eqref{SPLint}.
First, we need a boundary triplet for $S^*$. It is well known
 \cite[Example 15.3]{S12} that a boundary triplet $\{\Gi_e,\Gamma_0^e,\Gamma_1^e\}$ for $S_e^*$ is given by
\begin{align}
\label{summandInt}
\Gi_e:=\C^2,\quad
\Gamma_0^{(e)}\psi_e:=\begin{pmatrix}\psi_e(0+)\\ \psi_e(\ell(e)-)\end{pmatrix},\quad  \Gamma_1^{(e)}\psi_e:=\begin{pmatrix}
\psi'_e(0+)\\-\psi'_e(\ell(e)-)
\end{pmatrix}.
\end{align}

If $0<\inf_{e\in E} \ell(e)<\sup_{e\in E} \ell(e)<\infty$, then it follows from \cite{KM10}
that a boundary triplet for $S^*$ is given by the direct sum of the triplets \eqref{summandInt},
\begin{align}
\label{natural}
\{\Gi,\Gamma_0,\Gamma_1\}:=\Big\{\Eplus\Gi_e, \Eplus \Gamma_0^{(e)} ,\Eplus \Gamma_1^{(e)}\Big\}.
\end{align}
Each entry of an element of $\Gi$ corresponds to a vertex of the decoupled graph, i.e.\ the elements of $\Gi$ are sequences $(x_{(e,t)})_{(e,t)\in I}$ with $I:=E\times\{0,1\}$. For $\psi\in W^{2,2}(G)$ we write
\begin{align*}
\begin{split}
\Gamma_0\psi&:=(\Gamma_0^{(e,t)}\psi)_{(e,t)\in I}=(\psi_e(t\ell(e)))_{(e,t)\in I},\\ \Gamma_1\psi&:=(\Gamma_1^{(e,t)}\psi)_{(e,t)\in I}=(\sgn(e,t)\psi_e'(t\ell(e)))_{(e,t)\in I}.
\end{split}
\end{align*}
Using this boundary triplet, the condition $\psi\in C(G)$ in \eqref{PIintro} is equivalent to
\begin{align*}
(\Gamma_0^{(e,t)}\psi_e)_{(e,t)\in I_v}\in\Gi_v:=\Span\{1_v\},\quad  1_v:=(1,\ldots,1)\in\C^{|I_v|},
\end{align*}
for all $v\in V$.
Let $\deg v:=|I_v|$ be the \textit{degree} of $v\in V$. Here and in the following
we make the (crucial) assumption, that the graphs are \textit{locally finite}, i.e.
$$
\deg v < \infty \quad \mbox{for all } v \in V.
$$
The expressions in the equality in \eqref{PIintro} are equivalent to
\begin{align*}
P_{\Gi_v}(\Gamma_1^{(e,t)}\psi_e)_{(e,t)\in I_v}&=\frac{1}{\|1_v\|^2}((\Gamma_1^{(e,t)}\psi_e)_{(e,t)\in I_v},1_v)_{\C^{\deg v}}1_v\\
&=\frac{1}{\deg v}\sum_{(e,t)\in I_v}\sgn(e,t)\psi_e'(t\ell(e))\cdot 1_v\\
&=\frac{\alpha(v)}{\deg v}(\Gamma_0^{(e,t)}\psi_e)_{(e,t)\in I_v}
\end{align*}
Let $\iota_v$ be the natural embedding of elements of $\Gi_v$ in the sequence space $\Gi$.
For the operator
$$
L:=\oplus_{v\in V} L_v \quad \mbox{with } L_v\iota_v1_v:=\frac{\alpha(v)}{\deg v}\iota_v 1_v,\;  \dom L_v=\iota_v\Span\{1_v\}
$$
on $\Gi_{\mathcal{V}}:=\oplus_{v\in V}\iota_v \Gi_v$ we have
$$
S_L=H_{\alpha}
$$
in the case $0<\inf_{e\in E}\ell(e)<\sup_{e\in E}\ell(e)<\infty$.

In Proposition \ref{bekannt}, we show that the extension $S_L$ of $S$ is self-adjoint, semi-bounded from below and has discrete  spectrum if and only if the operator $L$ has this property. In our point interaction example the operator $L$ is just an infinite diagonal matrix, therefore the above mentioned spectral properties translate easily to $H_{\alpha}$, see \cite{EKMN17}.

If $\inf_{e\in E}\ell(e)=0$, then there is no natural candidate for a boundary triplet associated to $S^*$ since the operators in \eqref{natural} are in general not defined on $\dom S^*$. However, it was shown in \cite{KM10} that the triplet
 \eqref{natural} is a so called \textit{boundary relation} in the sense of \cite{DHMS06}.
To obtain a boundary triplet for $S^*$ from \eqref{natural} a regularization technique has been applied in \cite{CMP13,EKMN17,KM10,MN09,MN12}. 
Here we apply in Theorem \ref{handyreg} below the technique from \cite{CMP13} for operators where there  exists $\lambda_0\in\R$ and $\varepsilon>0$ such that $(\lambda_0-\varepsilon,\lambda_0+\varepsilon)\in\bigcap_{n=0}^{\infty}\rho\big(S_n^*|_{\ker\Gamma_0^{(n)}}\big)$. Then a (regularized) boundary triplet  $\left\{\widetilde{\Gi},\nplus\widetilde{\Gamma}_0^{(n)},\nplus\widetilde{\Gamma}_1^{(n)}\right\}$ is given by
\begin{align}
\label{regInt}
\widetilde{\Gi}:=\Gi,\quad  \widetilde{\Gamma}_0^{(n)}:=\sqrt{\|M_n'(\lambda_0)\|}\Gamma_0^{(n)},\quad \widetilde{\Gamma}_1^{(n)}:=\frac{\Gamma_1^{(n)}-M_n(\lambda_0)\Gamma_0^{(n)}}{\sqrt{\|M_n'(\lambda_0)\|}},
\end{align}
where $M_n$ is the Weyl function of the boundary triplet $\{\Gi_n,\Gamma_0^{(n)},\Gamma_1^{(n)}\}$.
Again, one can represent extensions of $S$  in terms of an operator $\widetilde{L}$
(now with respect to the regularized triplet  $\{\Gi,\widetilde\Gamma_0,\widetilde\Gamma_1\}$ from \eqref{regInt})
in the form of \eqref{SPLint},
\begin{equation}
\label{newrep}
S_{\widetilde{L}}=\{f\in\dom S^*~|~\widetilde{L}\widetilde{\Gamma}_0f=
P_{\widetilde{\Gi}_{\op}}\widetilde{\Gamma}_1 f,\
\widetilde{\Gamma}_0f\in\widetilde{\Gi}_{\op}\},
\end{equation}
where $\widetilde{L}$ is defined on some subspace $\widetilde{\Gi}_{op}$ of $\Gi$.
Whereas in the example above the operator $L$ is just an (infinite) diagonal operator,
now, in general, the operator $\widetilde{L}$ has a more complex structure.

The operator $\widetilde{L}$ from above, that describes the extensions with respect to the regularized boundary mappings,
is studied in \cite{CMP13,EKMN17,KM10}.
In \cite{KM10} Schr\"{o}dinger operators with point interactions on the real line are
 considered. In this case, roughly speaking, the operator $L$ in \eqref{newrep} for
 a point interaction if \eqref{natural} is a boundary triplet,
is a diagonal operator, whereas the operator $\widetilde L$
is a Jacobi operator and therefore a correspondence of extensions describing such interactions and Jacobi operators is made
in \cite{KM10}. In particular, criteria for self-adjointness, semi-boundedness from below and discreteness of the spectrum
are obtained from corresponding criteria for Jacobi operators. Later, in \cite{CMP13} the ideas of \cite{KM10} were extended to the case of Dirac operators with point interactions on the real line, so called Gesztesy-\v{S}eba realzations, see \cite{GS87}.
 Recently, in \cite{EKMN17} the regularization is applied to quantum graphs and Laplacians with point interactions are studied.
 In this case, the operator $\widetilde{L}$ in \eqref{newrep} is a discrete Laplacians on a weighted $\ell^2$-space, see \cite{KL10,KL11,GHKLW13} and the references therein.

Here we consider a more general class of extensions of symmetric direct sum operators $S=\nplus S_n$:
 \textit{locally finite extensions} $S_L^{\loc}$. It turns out that the operator $\widetilde{L}$ from above is also a weighted discrete Laplacian. The locally finite extensions of $S$ are such that they extend the quantum graph examples to more general structures.
 In particular, the symmetric operators $S_n$ may have an arbitrary but finite  defect indices

We study properties of the extensions $S_{L}^{\loc}$ like self-adjointness, semi-boundedness from below and discreteness of the spectrum
in terms of the associated weighted discrete Laplacian $\widetilde{L}$
 to the extension $S_L^{\loc}$. We show that self-adjointness, semi-boundedness from below and discreteness of the spectrum
 of $\widetilde{L}$  implies the same property for $S_{L}^{\loc}$.
 Sufficient conditions for such properties for $S_{L}^{\loc}$ are obtained recently in \cite{GHKLW13,KL11}.

In the case where \eqref{natural} is not a boundary triplet, some recent approaches \cite{P14,SSVW15} without using the regularization technique, lead to a parametrization of the self-adjoint extensions of $S$,
 but without explicit criteria for the above mentioned properties (like
 (self-adjointness, semi-boundedness from and discreteness of the spectrum).

Moreover, the boundary triplet approach to quantum graphs was previously applied in numerous works, see e.g.\ \cite{BL10,EKMN17,KS99,KS06,P06,P08}.
In \cite{BL10,KS99,KS06} finite graphs are considered. Graphs with an infinite number
of edges but with finite vertex degree were considered in \cite{P06}, under the assumption that $\ell(e)=1$ for all $e\in E$, and assuming that $\inf_{e\in E}\ell(e)>0$ in \cite{P08}.
The study of the operators $S_L$ was carried out in \cite{BL10} for star-graphs and for quantum graphs satisfying $\inf_{e\in E}\ell(e)>0$
in \cite{LSV14}.

The paper is organized as follows: First, we recall linear relations in Hilbert space and boundary triplets. From the boundary triplet theory, we collect some results on the properties of the extension $S_L$ given by \eqref{SPLint} which can be described terms of the operator $L$ and the Weyl function of an underlying boundary triplet for  $S^*$.  In Section \ref{sec:LFE} we introduce locally finite extension $S_L^{\loc}$ and construct an associated discrete Laplacian $D_L$ such that roughly speaking $S_{L}^{\loc}=S_{D_L}$ holds in the sense of \eqref{newrep} with $\widetilde{L}=D_L$. From this relation, we obtain conditions for the self-adjointness, lower semi-boundedness and discreteness of the spectrum of $S_L^{\loc}$. These conditions only depend on the matrices $L_v$, the subspaces $\Gi_v$ and the decoupled Weyl functions $M_n$.
Finally, in Section \ref{sec:Dirac} we apply our results to Dirac operators with point interactions on infinite graphs.

\section{Linear relations in Hilbert spaces}
 Let $(\Hi,(\cdot,\cdot)_{\Hi})$ be a separable Hilbert space. A \textit{(closed) linear relation in} $\Hi$ is a (closed) subspace of $\Hi\times\Hi$ and the set of all closed linear relations in $\Hi$ is denoted by $\widetilde{\mathcal{C}}(\Hi)$.
 For a linear operator $T$ defined in $\Hi$ with values in $\Hi$, the graph of $T$  is a linear relation in $\Hi$. 
 The set of all closed linear operators in $\Hi$ is denoted by $\mathcal{C}(\Hi)$. For the subspace of \textit{bounded} linear operators defined on $\Hi$ we write $\mathcal{L}(\Hi)$.

The \textit{domain}, the \textit{range}, the \textit{kernel}, the \textit{multivalued part} and the \textit{inverse} of a linear relation $\Theta$ in $\Hi$ are given by
\begin{align*}
    \dom\Theta&:=\{f\in\Hi ~|~ (f,f')\in\Theta~ \text{for some $f'\in\Hi$}\},\\
    \ran\Theta&:=\{f'\in\Hi ~|~ (f,f')\in\Theta~ \text{for some $f\in\Hi$}\},\\
    \ker\Theta&:=\{f\in\Hi ~|~ (f,0)\in\Theta\},\\
    \mul\Theta&:=\{f'\in\Hi ~|~ (0,f')\in\Theta\},\\
    \Theta^{-1}&:=\{(f',f)\in\Hi^2 ~|~(f,f')\in \Theta\}.
\end{align*}
Recall that the (operator-like) \textit{sum} of two linear relations $\Theta_1$ and $\Theta_2$ in $\Hi$ is given by
\[
\Theta_1+\Theta_2:=\{(f,f_1'+f_2')\in\Hi\times\Hi ~|~ (f,f_1')\in\Theta_1,\ (f,f_2')\in\Theta_2\}.
\]
Let $\Theta$ be a closed linear relation in $\Hi$. The set of all $\lambda\in\C$ such that  $(\Theta-\lambda)^{-1}$ is the graph of an operator from $\Lc(\Hi)$ is called  \textit{resolvent set} $\rho(\Theta)$ of $\Theta$. The complement of $\rho(\Theta)$ in $\C$ is the \textit{spectrum} $\sigma(\Theta)$ of $\Theta$.
The \textit{adjoint} $\Theta^*$  of a linear relation $\Theta$ in $\Hi$ is defined as
\[
\Theta^*:=\{(g,g')\in\Hi^2 ~|~ (f',g)_{\Hi}=(f,g')_{\Hi}\ \text{for all $(f,f')\in\Theta$}\}.
\]
A linear relation is called \textit{symmetric (self-adjoint)} if $\Theta\subseteq\Theta^*$ (resp.\ $\Theta=\Theta^*$).\\
For $\Theta\in\widetilde{\mathcal{C}}(\Hi)$ we have
\begin{align*}
\mul\Theta=(\dom \Theta^*)^{\perp},\quad \mul\Theta^*=\dom\Theta^{\perp}.
\end{align*}

Given a self-adjoint linear relation $\Theta$, we can associate a self-adjoint operator
 on the Hilbert space $\overline{\dom\Theta}$, see \cite[Theorem 5.3]{A61}. Below, we present a
 somehow converse result.
\begin{Proposition}
\label{LTheta}
Let $\Hi_{\op}$ be a closed subspace of a Hilbert space $\Hi$ and consider a densely defined operator $L$ from $\Hi_{\op}$ to $\Hi_{\op}$. Define
\begin{align}
 \label{DSrelation}
 \Theta_{L}:=\{(f,Lf+g) ~|~ f\in\dom L, \ g\in\Hi_{\op}^{\perp} \}\subseteq{\Hi\times\Hi}.
 \end{align}
 Then the following holds.
\begin{itemize}
\item[\rm (a)] We have $\Theta_{L}^*=\Theta_{L^*}$. If $L$ is closable, we have $\overline{\Theta_{L}}=\Theta_{\overline{L}}$.

\item[\rm (b)] $\Theta_L$ is closed (symmetric, self-adjoint) if and only if $L$ is closed (resp.\ symmetric, self-adjoint).

\item[\rm (c)] If $L$ is symmetric then all extensions $\widetilde \Theta$ with $\Theta_L\subseteq\widetilde\Theta\subseteq\Theta_L^*$ are of the form $\Theta_{\widetilde L}$, where $\widetilde L$ is an extension of $L$.

\item[\rm (d)] If $L$ is self-adjoint, then $\rho(L)=\rho(\Theta_{L})$ and for all $\lambda\in\rho(L)$
\[
(\Theta_{L}-\lambda)^{-1}=\begin{pmatrix}(L-\lambda)^{-1} & 0 \\ 0 & 0 \end{pmatrix}\in\mathcal{L}(\Hi_{\op}\oplus\Hi_{\op}^\perp).
\]
\end{itemize}
\end{Proposition}
\begin{proof}
Let $(f,g)\in\Theta_L^*$. Then for all $(f',Lf'+g')\in\Theta_L$
with $f'\in\dom L$ and $g'\in\Hi_{\op}^\perp$ we have
\begin{align}
\label{mvpweg}
(g,f')=(f,Lf'+g').
\end{align}
Choosing $f'=0$ we obtain $f\in \Hi_{\op}$. Therefore, we conclude from \eqref{mvpweg}
$$
(g,f')=(f,Lf')
$$
for all $f'\in \dom L$. This implies that $f\in\dom L^*$ and $P_{\Hi_{\op}}g= L^*f$. Hence,
$$
(f,g)= (f,P_{\Hi_{\op}}g + P_{\Hi_{\op}^\perp}g) =
(f,L^*f + P_{\Hi_{\op}^\perp}g)\in\Theta_{L^*}.
$$
Assume conversely that $(f,L^*f+g)\in\Theta_{L^*}$ with $f\in\dom L^*$ and some $g\in\Hi_{\op}^\perp$.
Then we have for all $(f',Lf'+g')\in\Theta_L$ with $f'\in\dom L$ and $g'\in\Hi_{\op}^\perp$
$$
(f',L^*f+g) = (f',L^*f)= (Lf',f)=(Lf'+g',f)
$$
and therefore $(f,L^*f+g)\in\Theta_L^*$.
Thus we have seen that $\Theta_L^*=\Theta_{L^*}$.

Let $(f,\overline L f+g)\in\Theta_{\overline{L}}$ with $f\in\dom \overline L$ then there is a sequence $((f_n,Lf_n))_{n\in\N}$
 which converges in $\Hi^2$ to $(f,\overline L f)$. But then $(f_n,Lf_n+g)\in\Theta_L$ converges in $\Hi^2$ to $(f,\overline{L}f+g)$,
 as $n\rightarrow\infty$, which implies  $(f,\overline{L}f+g)\in\overline{\Theta_L}$.

Conversely, let $(f,g')\in\overline{\Theta_L}$ then there exists a sequence $(f_n,Lf_n+g_n)\in\Theta_L$ with $g_n\in\Hi_{\op}^\perp$ and $f_n\in\dom L$ which converges to $(f,g')$. As $\Hi^2= \left(\Hi_{\op}\times \Hi_{\op}\right)\oplus
\left(\Hi_{\op}^\perp \times \Hi_{\op}^\perp\right)$ we have $(f_n,Lf_n)\rightarrow (f,P_{\Hi_{\op}}g')$ and
$(0,g_n) \to (0,P_{\Hi_{\op}^\perp}g')$, as $n\rightarrow\infty$. Therefore $f\in\dom\overline{L}$
with  $\overline L f= P_{\Hi_{\op}}g'$. Hence
$$
(f,g')=(f, \overline L f+P_{\Hi_{\op}^\perp}g') \in \Theta_{\overline{L}}.
$$

The assertion (b) is a consequence of (a). We show (c).
Let $\widetilde\Theta$ be an extension of $\Theta_L$ with
$\Theta_L\subseteq\widetilde\Theta\subseteq\Theta_{L^*}$. Obviously,
$$
\dom L =\dom\Theta_L\subseteq \dom\widetilde\Theta\subseteq\dom\Theta_{L^*}=\dom L^*.
$$
Set $\widetilde L:=L^*|\dom\widetilde \Theta$. Then $\widetilde L$ is an extension of $L$. As $\widetilde \Theta\subset
 \Theta_{L^*}$, every element $(f,g') \in \widetilde \Theta$ satisfies $f\in \dom \widetilde \Theta\subset
 \dom L^*$ and has a representation
 $$
 (f,g') = (f, L^*f+g)
 $$
 for some $g \in \Hi_{\op}^\perp$. As $L^*f = \widetilde Lf$ for $f\in \dom\widetilde\Theta$,
 $(f,g') \in \Theta_{\widetilde L}$ follows. Hence, $\widetilde \Theta \subset  \Theta_{\widetilde L}$.
 The converse inclusion is obvious and (c) is shown.

 For the last statement observe that we have for all $\lambda \in \mathbb C$
 $$
 (\Theta_L-\lambda)^{-1} = \left\{
 ((L-\lambda)f+g,f) ~|~ f\in\dom L, \ g\in\Hi_{\op}^{\perp}
 \right\}.
 $$
From this (d) follows easily.
\end{proof}

\section{Extension theory of symmetric operators with boundary triplets}

We review the boundary triplet theory following \cite{DM91}, see also \cite{K75}.

\begin{Definition}
For a densely defined symmetric operator $A\in\mathcal{C}(\Hi)$ in a Hilbert space $\Hi$ we say that $\{\Gi,\Gamma_0,\Gamma_1\}$ is a {\rm boundary triplet} for $A^*$ if $(\Gi,(\cdot,\cdot)_{\Gi})$ is a Hilbert space, $(\Gamma_0,\Gamma_1)^\top:\dom A^*\rightarrow \Gi^2$ is surjective and the following abstract Green identity holds
\begin{align}
\label{GI}
(A^*f,g)_{\Hi}-(f,A^*g)_{\Hi}=(\Gamma_1f,\Gamma_0g)_{\Gi}-(\Gamma_0f,\Gamma_1 g)_{\Gi}.
\end{align}
\end{Definition}
Boundary triplets are a standard tool to describe all closed extensions of a given symmetric operator. For a densely defined symmetric operator $A\in\mathcal{C}(\Hi)$, we fix a boundary triplet $\{\Gi,\Gamma_0,\Gamma_1\}$ for $A^*$. The extension $A_{\Theta}$ of $A$ corresponding to a parameter $\Theta\in\widetilde{\mathcal{C}}(\Gi)$ is defined as
\begin{align*}
\dom A_{\Theta}:=\{f\in\dom A^* ~|~ (\Gamma_0f,\Gamma_1f)\in\Theta \},\quad A_{\Theta}f:=A^*f.
\end{align*}
The correspondence between the closed linear relations
$\Theta\in\widetilde{\mathcal{C}}(\Gi)$ and the closed extensions $A_{\Theta}$ of $A$ is bijective (see, e.g., \cite{DM91}).
The following two special self-adjoint extensions of $A$ will play a prominent role:
$$
A_0:=A_{\{0\}\times \Gi}=A^*|_{\ker\Gamma_0}\quad \mbox{and}\quad A_1:=A_{\Gi\times\{0\}}=A^*|_{\ker\Gamma_1}
$$
In \cite{DM91} a correspondence of properties between $\Theta\in\widetilde{\mathcal{C}}(\Gi)$ and $A_{\Theta}\in\mathcal{C}(\Gi)$ was established using the concept of the \textit{$\gamma$-field} and the \textit{Weyl function}.
\[
\gamma:\rho(A_0)\rightarrow\mathcal{L}(\Hi,\Gi),~ \gamma(\lambda):=(\Gamma_0|_{\mathcal{N}_{\lambda}})^{-1},\quad \mathcal{N}_{\lambda}(A):=\{f\in\dom A^* ~|~ A^*f=\lambda f\},
\]
\[
M:\rho(A_0)\rightarrow\mathcal{L}(\Gi),\quad M(\lambda):=\Gamma_1\gamma(\lambda).
\]

Here we prefer the following description of the extensions.
Let $L$ be a densely defined operator on a subspace $\Gi_{\op}$ of $\Gi$ mapping into $\Gi_{\op}$. We consider the relation $\Theta_{L}$ from \eqref{DSrelation} and the associated extension $A_{L} :=A_{\Theta_{L}}$ and therefore
\begin{align}
\label{Apl}
\dom A_{L}=\{f\in \dom A^* ~|~\Gamma_0f\in\dom L,~ L\Gamma_0f=P_{\overline{\dom L}}\Gamma_1f\},
\end{align}
where $P_{\overline{\dom L}}$ is the orthogonal projection onto $\Gi_{\op}=\overline{\dom L}$
since $L$ is assumed to be densely defined in $\Gi_{\op}$.
Proposition \ref{LTheta} and some well known results on the relationship between $\Theta\in\widetilde{\mathcal{C}}(\Gi)$ and $A_{\Theta}\in\mathcal{C}(\Hi)$
from \cite{DM91, KM10} lead to the next statement. Here we use the notation $\mathfrak{S}_p(\Hi)$ with  $p\in(0,\infty]$ for the two sided \textit{Schatten-von Neumann ideal} and we denote by $n_{\pm}(A):=\dim\mathcal{N}_{\pm i}(A)$ the \textit{defect numbers} of a symmetric densely defined linear operator $A$.
\begin{Proposition}
\label{bekannt}
Let $A$ be a densely defined symmetric operator in $\Hi$ with boundary triplet $\{\Gi,\Gamma_0,\Gamma_1\}$ for $A^*$ and let $L$ be a densely defined operator in a subspace $\Gi_{\op}$ of $\Gi$ then the following holds.
\begin{itemize}
\item[\rm (a)] $A_{L}$ is self-adjoint (symmetric) if and only if $L$ is self-adjoint (resp.\ symmetric).

\item[\rm (b)] $A_{\overline{L}}=\overline{A_{L}}$, $A_{L^*}=A_{L}^*$ and $n_{\pm}(A_{L})=n_{\pm}(L)$.

\item[\rm (c)] If $L$ is symmetric, then there is a bijective correspondence between the extensions of $L$ and the extensions of $A_L$.

\item[\rm (d)] For $\lambda\in\rho(A_0)$ we have $\lambda\in\rho(A_{L})$ if and only if  $0\in\rho(\Theta_{L}-M(\lambda))$. In this case the Krein resolvent formula holds
	\begin{align*}
(A_{L}-\lambda)^{-1}-(A_{0}-\lambda)^{-1}=\gamma(\lambda)(\Theta_{L}-M(\lambda))^{-1}\gamma(\overline{\lambda})^*.
\end{align*}
\item[\rm (e)] Let $(A_0-\lambda_0)^{-1}\in\mathfrak{S}_{p}(\Hi)$ for some  $\lambda_0\in\rho(A_0)$ and $p\in[1,\infty]$. Then  $(A_{L}-\lambda)^{-1}\in\mathfrak{S}_p(\Hi)$ if and only if  $(L-\lambda)^{-1}\in\mathfrak{S}_p(\Gi)$ for  $\lambda\in\rho(L)$.
\end{itemize}
\end{Proposition}

Let $A$ be a densely defined symmetric operator which is semi-bounded from below, i.e.\ $A\geq\gamma$ for some $\gamma\in\R$. Then there is a distinguished, in some sense maximal, semi-bounded self-adjoint extension $A_F\geq \gamma$, which is called the \textit{Friedrichs extension} of $A$, see e.g.\ \cite[Section 10.4]{S12}.

Given boundary triplet $\{\Gi,\Gamma_0,\Gamma_1\}$ of $A^*$ with Weyl function $M$ such that $A_0$ equals the Friedrichs extension $A_F$, then we use the notation
$M(\lambda)\rightrightarrows-\infty$ for $\lambda\rightarrow-\infty$ to indicate that for any $\gamma>0$
 there exists $\lambda_{\gamma}$ with $-M(\lambda_{\gamma})\geq \gamma$.

We collect some results on nonnegative extensions from \cite{DM91,DM95}, see also \cite{S12}.
\begin{Proposition}
\label{negEWundso}
Given a densely defined symmetric operator $A\in\mathcal{C}(\Hi)$, a boundary triplet $\{\Gi,\Gamma_0,\Gamma_1\}$ for $A^*$ with $A_0=A_F\geq\gamma$ for $\gamma>0$ and a self-adjoint operator $L\in\widetilde{\mathcal{C}}(\Gi_{\op})$ on a subspace $\Gi_{\op}$ of $\Gi$. Then the following holds.
\begin{itemize}

\item[\rm (a)] $L-P_{\Gi_{\op}}M(\lambda_0)|_{\Gi_{\op}}\geq0$ for $\lambda_0<\gamma$ implies $A_{L}\geq\lambda_0$.

\item[\rm (b)] If $M(\lambda)\rightrightarrows-\infty$ for $\lambda\rightarrow-\infty$ then  $A_{L}$ is semi-bounded from below if and only if $L$ is semi-bounded from below.

\end{itemize}
\end{Proposition}

In the lemma below, we decribe the change of a boundary triplet $\{\Gi,\Gamma_0,\Gamma_1\}$ under unitary transformations of the space $\Gi$.

\begin{Lemma}
\label{unitary}
Let $A\in\mathcal{C}(\Hi)$ be a densely defined symmetric operator with a boundary triplet $\{\Gi,\Gamma_0,\Gamma_1\}$ for $A^*$ and a unitary operator $U:\Gi\rightarrow\widehat{\Gi}$ then $\{\widehat{\Gi}, U\Gamma_0, U\Gamma_1\}$ is a boundary triplet for $A^*$ with Weyl function $\lambda\mapsto U M(\lambda)U^*$ on  $\rho(A^*|_{\ker\Gamma_0})$. Furthermore the extension $A_{L}$ given by \eqref{Apl}  can be written with $\widehat L:=ULU^*$ as
\[
\dom A_{L}=\{f\in\dom A^*~|~U\Gamma_0f\in \dom \widehat L,\quad \widehat LU\Gamma_0f=P_{\overline{\dom \widehat L}}U\Gamma_1f\}.
\]
\end{Lemma}
\begin{proof}
Since $U$ is unitary, the mapping $f\mapsto (U\Gamma_0f, U\Gamma_1f)$ from $\dom A^*$ into $\widehat{\Gi}^2$ is onto and the abstract Green identity \eqref{GI} holds. Hence $\{\widehat{\Gi},U\Gamma_0,U\Gamma_1\}$ is a boundary triplet for $A^*$ with $A^*|_{\ker\Gamma_0}=A^*|_{\ker U\Gamma_0}$ and Weyl function $\lambda\mapsto UM(\lambda)U^*$ which is defined for all $\lambda\in\rho(A^*|_{\ker\Gamma_0})$.
Given that $f\in\dom A_{L}$ then we have $\Gamma_0f\in\dom L$ and $L\Gamma_0f=P_{\overline{\dom L}}\Gamma_1f$ which is equivalent to
\begin{align}
\label{eins}
U \Gamma_0f\in U\dom L,\quad  ULU^* U\Gamma_0f=UP_{\overline{\dom L}}U^*U\Gamma_1f.
\end{align}
Furthermore, it is easy to see that
\begin{align}
\label{zwei}
\dom \widehat{L}=\dom ULU^*=\dom LU^*=U\dom L.
\end{align}
Moreover $UP_{\overline{\dom L}}U^*$ is an orthogonal projection, satisfying
\begin{align}
\label{drei}
UP_{\overline{\dom L}}U^*=P_{U\overline{\dom L}}=P_{\overline{U\dom L}}=P_{\overline{\dom\widehat L}}.
\end{align}
Rewriting \eqref{eins} with \eqref{zwei} and \eqref{drei} completes the proof of the lemma.
\end{proof}

\section{Locally finite extensions of direct sums of symmetric operators}
\label{sec:LFE}
In this section, we introduce direct sum operators and their locally finite extensions.
Throughout this section we consider a family of Hilbert spaces $\{\Hi_n\}_{n\in\N}$
with inner product $(\cdot,\cdot)_{\Hi_n}$ and densely defined symmetric operators $S_n\in\mathcal{C}(\Hi_n)$ with boundary triplets  $\{\Gi_n,\Gamma_0^{(n)},\Gamma_1^{(n)}\}$ for $S_n^*$ such that $\dim\Gi_n<\infty$, $n\in\N$.
We introduce the direct sum Hilbert space $\Hi$,
\[
\Hi:=
\nplus\Hi_n:=\{x=(x_n)_{n\in\N} : x_n\in\Hi_n, (x,x)_{\Hi}<\infty \}
\]
with inner product
\[
 ((x_n)_{n\in\N},(y_n)_{n\in\N})_{\Hi}:=\sum_{n=0}^{\infty}(x_n,y_n)_{\Hi_n}.
\]
Acting on $\Hi$ we introduce the direct sum operator $S:=\nplus S_n$ via
\begin{align*}
\dom S:=\bigg\{(f_n)_{n\in\N} ~\Big|~ f_n\in\dom S_n,\ \sum_{n=0}^{\infty}\|S_nf_n\|_{\Hi_n}^2<\infty\bigg\},\quad S(f_n)_{n\in\N}:=(S_nf_n)_{n\in\N}.
\end{align*}
The case of a finite dimensional direct sum
Hilbert space $\Hi$ is obtained by setting $\Hi_n:=\{0\}$ and $S_n:=0$ for all $n\geq N$ and some $N\in\N$.
It is easy to see that $S$ is densely defined, closed with the adjoint
\[
\left(\nplus S_n\right)^*=\nplus S_n^*.
\]
Since $S_n\subseteq S_n^*$ for all $n\in\N$ it is easy to see that $S$ is symmetric with $n_{\pm}(S)=\sum_{n=0}^{\infty}n_{\pm}(S_n)$. To describe the extensions of $S$, the natural candidate for a boundary triplet
for $S^*$ is given by $\Gi:=\nplus\Gi_n$ with the boundary mappings $\Gamma_i$, $i=1,2$,
\begin{align*}
\dom\Gamma_i:=\left\{(f_n)_{n\in\N} : f_n\in\dom\Gamma_i^{(n)},\ \sum_{n=0}^\infty \|\Gamma_i^{(n)}f_n\|^2<\infty\right\},\quad  \Gamma_i(f_n)_{n\in\N}:=\left(\Gamma_i^{(n)}f_n\right)_{n\in\N}
\end{align*}
which can also be written in the form
\begin{align}
\label{dstriplet}
\Gi:=\nplus\Gi_n,\quad \Gamma_0:=\nplus\Gamma_0^{(n)},\quad \Gamma_1:=\nplus\Gamma_1^{(n)}.
\end{align}
In general, the operators $\Gamma_0$ and $\Gamma_1$ are only defined on a subspace of $\dom S^*$ such that \eqref{dstriplet} is not a boundary triplet for $S^*$. However, it was shown in \cite{KM10} that the triplet  $\{\Gi,\Gamma_0,\Gamma_1\}$ given by \eqref{dstriplet}
forms a single valued boundary relation in the sense of \cite{DHMS06}.

We use a particular regularization from \cite{CMP13} for the direct sum triplet \eqref{dstriplet} for operators with a common real point in the resolvent set, i.e. we assume that for $S_{n0}:=S_n^*|_{\ker\Gamma_0^{(n)}}$ there exist $\lambda_0\in\R$ and $\varepsilon>0$ such that $(\lambda_0-\varepsilon,\lambda_0+\varepsilon)\subseteq\bigcap_{n\in\N}\rho(S_{n0})$. In the theorem below we use \cite[Theorem 2.12]{CMP13}, to provide a boundary triplet for the direct sum operator $S^*$.
\begin{Theorem}
\label{handyreg}
Let $\{S_n\}_{n\in\N}$ be a family of densely defined symmetric linear operators $S_n\in\mathcal{C}(\Hi_n)$ with boundary triplets $\{\Gi_n,\Gamma_0^{(n)},\Gamma_1^{(n)}\}$ for $S_n^*$ and Weyl functions $M_n$
and $(\lambda_0-\varepsilon,\lambda_0+\varepsilon)\subseteq\bigcap_{n=0}^{\infty}\rho(S_{n0})$ for some $\varepsilon>0$ and $\lambda_0\in\R$.
Then $\left\{\nplus\Gi_n,\nplus\widetilde{\Gamma}_0^{(n)},\nplus\widetilde{\Gamma}_1^{(n)}\right\}$ with
\begin{align}
\label{gutestripel}
\widetilde{\Gamma}_0^{(n)}:=\sqrt{\|M_n'(\lambda_0)\|}\Gamma_0^{(n)},\quad \widetilde{\Gamma}_1^{(n)}:=\sqrt{\|M_n'(\lambda_0)\|}^{-1}\left(\Gamma_1^{(n)}-M_n(\lambda_0)\Gamma_0^{(n)}\right)
\end{align}
is a boundary triplet for $S^*=\nplus S_n^*$. The Weyl function $\widetilde{M}$ of this triplet is given by $\widetilde{M}:\rho(S^*|_{\ker\Gamma_0})\rightarrow\mathcal{L}(\Gi)$, $\lambda\mapsto\nplus\widetilde{M}_n(\lambda)$ with
\begin{align}
\label{tildeM}
\widetilde{M}_n:=\frac{1}{\|M_n'(\lambda_0)\|}(M_n-M_n(\lambda_0)).
\end{align}
\end{Theorem}
The construction of this regularization implies that  $S_n^*|_{\ker\Gamma_0^{(n)}}=S_n^*|_{\ker\widetilde{\Gamma}_0^{(n)}}$ and therefore
\begin{align}
\label{invariant}
S^*|_{\ker\widetilde{\Gamma}_0}=\nplus S_{n0}.
\end{align}
The remainder of this section is devoted to locally finite extensions. We assume that the Hilbert space $\Gi$ is given as the direct sum Hilbert space
$$
\Gi=\nplus\Gi_{n}=\nplus\C^{d_n} \quad \mbox{with } d_n<\infty \mbox{ and }  \Gi_{n}=\C^{d_n}.
$$
The elements of $\Gi$ are sequences of the form $x=(x_i)_{i\in I}$ where
\[I:=\{(n,d)~|~n\in\N,\ d=1,\ldots,d_n\}.\]

In the following we will consider a partition of $I$ into subsets $I_v$ where $v$ is an element of a countable index set $V$ such that the following conditions are fulfilled:
\begin{itemize}
\item[\rm (i)] $|I_v|<\infty$,

\item[\rm (ii)] $I_v\cap I_w=\emptyset$ for all $v,w\in V$ with $v\neq w$,

\item[(iii)] $\bigcup_{v\in V} I_v=I$.
\end{itemize}
Since $\Gamma_i^{(n)}f_n\in\C^{d_n}$ the sequence $(\Gamma_i^{(n)}f_n)_{n\in\N}$, $i=0,1$ is an element of $\times_{n=0}^{\infty}\C^{d_n}$, but not necessarily of the Hilbert space $\nplus\C^{d_n}$. Thus, we can view it as sequence $(\Gamma_i^{(n,d)}f_n)_{(n,d)\in I}$ where
\[
\Gamma_i^{(n,d)}f_n:=(\Gamma_i^{(n)}f_n)_d,\quad 1\leq d\leq d_n,\quad i=0,1
\]
is the $d$-th entry of $\Gamma_i^{(n)}f_n$. With this we introduce for $f\in \dom S^*$
\[
\Gamma_i^vf:=(\Gamma_i^{(n,t)}f_n)_{(n,t)\in I_v}, \quad i=0,1.
\]
Before we continue with the definition of locally finite extensions, we illustrate the definitions from above with the quantum graph example from the introduction.
\begin{Example}\label{SinPijama}
Consider the densely defined symmetric operators $S_1,\ldots,S_N$
with domains $\dom S_n:=W^{2,2}_0(0,\ell(e_n))$, $n=1,\ldots,N$  with $S_n\psi_n:=-\psi_n''$.
Then a boundary triplet for $S_n^*$ with $n=1,\ldots,N$ is given by
\[
\{\C^2,(\psi_n(0+),\psi_n(\ell(e_n)-))^\top,(\psi_n'(0+),-\psi_n'(\ell(e_n)-))^\top\},
\]
Hence $d_n=2$ for all $n$ and therefore
\[
I=\{1,\ldots,N\}\times\{1,2\}.
\]
Consider the index set $V=\{v_1,\ldots,v_{N+1}\}$. We introduce $I_{v_i}:=\{(i,1)\}$ for $i=1,\ldots,N$
and  $I_{v_{N+1}}:=\{(i,2) : i=1,\ldots,N\}$. It is easy to see that the conditions (i)-(iii) from above are satisfied. For each index $i=1,\ldots,N+1$ there is an edge associated with it and the sets $I_{v_i}$ describe which edges are glued together at the vertex $v_i$ which leads to a graph. In this simple example all vertices $v_i$, $i=1,\ldots,N$ corresponds so singleton sets $I_{v_i}$, i.e.,
only one vertex leads to $v_i$, whereas  in $v_{N+1}$ we have $N$ vertices. Hence the underlying
graph is a star graph with $N+1$ vertices and $N$ edges. Furthermore, we have
\begin{align*}
\Gamma_0^{v_n}(\psi_j)_{j=1}^{N}&=\psi_n(0+),\ \Gamma_1^{v_n}(\psi_j)_{j=1}^{N}=\psi_n'(0+), \quad  n=1,\ldots,N,\\[1ex]
\Gamma_0^{v_{N+1}}(\psi_j)_{j=1}^{N}&=(\psi_1(\ell(e_1)-),\psi_2(\ell(e_2)-),\ldots,\psi_N(\ell(e_N)-))^\top,\\[1ex]
\Gamma_1^{v_{N+1}}(\psi_j)_{j=1}^{N}&=(-\psi_1'(\ell(e_1)-),-\psi_2'(\ell(e_2)-),\ldots,-\psi_N'(\ell(e_N)-))^\top.
\end{align*}
Obviously, one easily can construct examples with infinitely many vertices and edges. Observe, as we only consider
locally finite extensions, that always  $|I_v|<\infty$ holds, which means, in the cases of graph-like constructions,
that in each edge there are only finitely many vertices.
\end{Example}
\begin{Example}
Here we give an example for a star graph with finite edges and vertices but with infinite edge length.
Consider the densely defined symmetric operators $S_1,\ldots,S_N$ from Example \ref{SinPijama}
and, in addition, $\dom S_{N+1}:=W^{2,2}_0(0,\infty)$ with $S_{N+1}\psi_{N+1}:=-\psi_{N+1}''$.
Then a boundary triplet for $S_n^*$ with $n=1,\ldots,N$ is given as in  Example \ref{SinPijama}
and a triplet for $S_{N+1}^*$ is given by  $\{\C,\psi_{N+1}(0+),\psi_{N+1}'(0+)\}$, see e.g.\ \cite[Example 15.5]{S12}.
Hence $d_n=2$ for all $n=1,\ldots,N$ but $d_{N+1}=1$ and therefore
\[
I=(\{1,\ldots,N\}\times\{1,2\})\cup\{(N+1,1)\}.
\]
Consider the index set $V=\{v_1,\ldots,v_{N+1}\}$. We introduce $I_{v_i}:=\{(i,1)\}$ for $i=1,\ldots,N$, $I_{v_{N+1}}:=\{(i,2) : i=1,\ldots,N\}\cup\{(N+1,1)\}$. As above we have a star graph, but with one vertex less, as the edge corresponding
to $N+1$ is a semi-axis,
\begin{align*}
\Gamma_0^{v_n}(\psi_j)_{j=1}^{N+1}&=\psi_n(0+),\ \Gamma_1^{v_n}(\psi_j)_{j=1}^{N+1}=\psi_n'(0+), \quad  n=1,\ldots,N,\\[1ex]
\Gamma_0^{v_{N+1}}(\psi_j)_{j=1}^{N+1}&=(\psi_1(\ell(e_1)-),\psi_2(\ell(e_2)-),\ldots,\psi_N(\ell(e_N)-), \psi_{N+1}(0+))^\top,\\[1ex]
\Gamma_1^{v_{N+1}}(\psi_j)_{j=1}^{N+1}&=(-\psi_1'(\ell(e_1)-),-\psi_2'(\ell(e_2)-),\ldots
,-\psi_N'(\ell(e_N)-),\psi_{N+1}'(0+))^\top.
\end{align*}
Similarly, one can construct graphs with infinitely many vertices and edges. Moreover, we stress that we are able
to allow $d_n>2$ with leads to structures which do no longer allow an interpretation as a graph.
\end{Example}
Now let $\Gi_v$ be a subspace of $\C^{|I_v|}$ and consider the Hermitian matrix $L_v:\Gi_v\rightarrow\Gi_v$.
We introduce the locally finite extension $S_{L}^{\loc}$ of $S$
\begin{align*}
\begin{split}
\dom S_{L}^{\loc}&:=\left\{f\in\dom S^*~|~L_v\Gamma_0^vf=P_{\Gi_v}\Gamma_1^vf,\ \Gamma_0^v f\in\Gi_v,\ v\in V\right\},\\
 S_{L}^{\loc}f&:=S^*f.
\end{split}
\end{align*}
It is shown in Proposition \ref{adjoint} below that $S_{L}^{\loc}$ is the adjoint of the operator $S_{L}^{\min}\subseteq S^*$ with
\[
\dom S_{L}^{\min}:=\left\{f\in \dom S_{L}^{\loc} ~|~ \supp (\Gamma_0^vf)_{v\in V}, \supp (P_{\Gi_v}\Gamma_1^vf)_{v\in V}\ \text{finite}\right\}
\]
where we used the \textit{support} of a sequence $x=(x_i)_{i\in I}\in\C^I$ given by
\[
\supp x:=\{i\in I~|~ x_{i}\neq0\}.
\]

For its proof we need a variant of the abstract Green identity \eqref{GI}.
\begin{Lemma}
\label{hilfslemma}
Let $f,g\in\dom S^*$ then
\begin{align}
\label{truncGI}
(S^*f,g)-(f,S^*g)=\sum_{v\in V}(\Gamma_1^vf,\Gamma_0^vg)-(\Gamma_0^vf,\Gamma_1^vg).
\end{align}
 Furthermore, given $v\in V$,  $y_0\in\Gi_v$ and $y_1\in \Gi_v^{\perp}$ there exists $g=(g_n)_{n\in\N}\in \dom S_L^{\min}$ with
  finite support such that the following equations hold
\begin{align}
\label{glsystem}
\begin{split}
&\Gamma_0^vg=y_0,\quad \Gamma_1^vg=y_1+L_vy_0,\\  &\Gamma_0^wg=\Gamma_1^wg=0,
\quad\text{for all $w\in V\setminus\{v\}$.}
\end{split}
\end{align}
\end{Lemma}
\begin{proof}
First, we show that for all $f=(f_n)_{n\in\N},g=(g_n)_{n\in\N}\in\dom S^*$, the sum $\sum_{n=0}^\infty(S_n^*f_n,g_n)$ converges absolutely. From Cauchy-Bunjakowski and H\"{o}lder inequality, we have
\[
\sum_{n=0}^\infty|(S_n^*f_n,g_n)|\leq\sum_{n=0}^\infty\|S_n^*f_n\|\|g_n\|\leq \|S^*f\|\|g\|<\infty.
\]
Next, using the abstract Green identity \eqref{GI} for the operators $S_n^*$ and changing the order of summation leads to
\begin{align*}
(S^*f,g)-(f,S^*g)&=\sum_{n=0}^{\infty}(S_n^*f_n,g_n)-(f_n,S_n^*g_n)\\ &=\sum_{n=0}^{\infty}(\Gamma_1^{(n)}f_n,\Gamma_0^{(n)}g_n)-(\Gamma_0^{(n)}f_n,\Gamma_1^{(n)}g_n)\\
&=\sum_{v\in V}(\Gamma_1^vf,\Gamma_0^vg)-(\Gamma_0^vf,\Gamma_1^vg),
\end{align*}
where the last equality follows from $\bigcup_{v\in V}I_v=I$.

For the proof of the second assertion we construct $g=(g_n)_{n\in\N}\in\dom S_{L}^{\min}$ satisfying the equations \eqref{glsystem}. Consider $n\in\N$ and the set $I_v$.  Given that $(n,d)\notin I_v$ for all $d=1,\ldots, d_n$ then we set $g_n:=0$. For $(n,d)\in I_v$, for some $d=1,\ldots, d_n$, the surjectivity of $(\Gamma_0^{(n)},\Gamma_1^{(n)})^\top:\dom S_n^*\rightarrow\Gi_n\times\Gi_n$ for all $n\in\N$ implies that we can choose $g_n$ such that the first and second equation in \eqref{glsystem} hold. From the construction we also have the lower system of equations in \eqref{glsystem} hold.
\end{proof}

Next, we show that $S_L^{\loc}$ is the adjoint of $S_L^{\min}$.
\begin{Proposition}
\label{adjoint}
We have $S_{L}^{\loc}=(S_{L}^{\min})^*$, in particular $S_{L}^{\loc}$ is closed.
\end{Proposition}
\begin{proof}
Let $f\in (S_{L}^{\min})^*$ then we have from \eqref{truncGI} for all $g\in\dom S_{L}^{\min}$
\begin{align}
\label{GIapplied}
0=(S^*f,g)-(f,S^*g)=\sum_{v\in V}(\Gamma_1^vf,\Gamma_0^vg)-(\Gamma_0^vf,\Gamma_1^vg).
\end{align}
For this equation we use \eqref{glsystem} from Lemma \ref{hilfslemma} with $y_0=0$ and $y_1\in\Gi_v^\perp$ which leads to $(\Gamma_0^vf,y_1)=0$. Since $y_1$ was arbitrary, we conclude that $\Gamma_0^vf\in\Gi_v$ for all $v\in V$.  Choose $g\in\dom S_{L}^{\min}$ that solves \eqref{glsystem} for $y_1=0$ and arbitrary $y_0\in\Gi_v$. With \eqref{GIapplied} this leads to
\[
0=(\Gamma_1^vf,y_0)-(\Gamma_0^vf,L_vy_0)=(P_{\Gi_v}\Gamma_1^vf-L_v\Gamma_0^vf,y_0).
\]
Since $y_0\in\Gi_v$ was arbitrary, we see $P_{\Gi_v}\Gamma_1^vf=L_v\Gamma_0^v$ for all $v\in V$ this proves $f\in\dom S_{L}^{\loc}$.

Assume conversely that $f\in\dom S_{L}^{\loc}$. For all $g\in\dom S_{L}^{\min}$ we have
\begin{align*}
\sum_{v\in V}(\Gamma_1^vf,\Gamma_0^vg)-(\Gamma_0^vf,\Gamma_1^vg)
&=\sum_{v\in V}(P_{\Gi_v}\Gamma_1^vf,\Gamma_0^vg)-(\Gamma_0^vf,P_{\Gi_v}\Gamma_1^vg)\\
&=\sum_{v\in V}(L_v\Gamma_0^vf,\Gamma_0^vg)-(\Gamma_0^vf,L_v\Gamma_0^vg)\\ &=0
\end{align*}
which implies with \eqref{truncGI}, $f\in\dom (S_L^{\min})^*$.
\end{proof}

We prove the main theorem of this section that allows us to describe
the extension $S_{L}^{\loc}$ with operators on $\ell^2(\widehat{V})$
for a countable index set $\widehat{V}$. For this we use the notation
\begin{align*}
C(\widehat{V}):=\{(f_v)_{v\in\widehat{V}}\in\ell^2(\widehat{V}) ~|~ \text{$\supp f$\ finite}\}.
\end{align*}
Furthermore, for the subspaces $\Gi_v$ of $\C^{|I_v|}$ we use the canonical embedding
\begin{align*}
\begin{split}
\iota_v:\Gi_v\rightarrow\oplus_{n\in\N}\C^{d_n},\quad (x_{(n,d)})_{(n,d)\in I_v}\mapsto(y_{(n,d)})_{(n,d)\in I},\\ y_{(n,d)}:=\begin{cases}x_{(n,d)}, & \text{if $(n,d)\in I_v$,}\\ 0, & \text{otherwise.}\end{cases}
\end{split}
\end{align*}
Therefore $\iota_v(\Gi_v)$ is a subspace of $\Gi$ and we have an orthogonal sum decomposition
\begin{align}
\label{defGV}
\Gi_{\Vc}:=\bigoplus_{v\in V}\iota_v\Gi_v.
\end{align}
In the following, we consider an orthogonal basis $\{b_w\}_{w\in\widehat{V}}$ of the subspace $\Gi_{\Vc}$, which has the property that each $b_w$ is an element of an orthogonal basis for some $\Gi_v$ and $\widehat{V}$ is a countable set of indices.
In the theorem below we will make use of the unitary operator $U:\Gi_{\Vc}\rightarrow\ell^2(\widehat V)$ given by $b_w\mapsto\|b_w\|e_w$.

\begin{Theorem}
\label{dieDS}
Let $\{S_n\}_{n\in\N}$ be a family of densely defined symmetric linear operators $S_n\in\mathcal{C}(\Hi_n)$ with boundary triplets $\{\Gi_n,\Gamma_0^{(n)},\Gamma_1^{(n)}\}$ for $S_n^*$ and Weyl functions $M_n$
and $(\lambda_0-\varepsilon,\lambda_0+\varepsilon)\subseteq\bigcap_{n=0}^{\infty}\rho(S_{n0})$ for some $\varepsilon>0$ and $\lambda_0\in\R$.
Consider $S_L^{\loc}$ with Hermitian matrices $L_v$, subspaces $\Gi_v$, $\Gi_{\Vc}$ given by \eqref{defGV} with orthogonal basis $\{b_w\}_{w\in\widehat{V}}$
and the operator $L=\oplus_{v\in V} L_v$ on $\Gi_{\Vc}$. Then the following holds.
\begin{itemize}
\item[\rm (a)]
The operator $L_{\min}$ in $\ell^2(\widehat{V})$ with $\dom L_{\min}=C(\widehat{V})$ given
as an infinite matrix operator,
\begin{align*}
L_{\min}:=\left(\frac{((L-\plus M_n(\lambda_0))b_v,b_w)}{\|Rb_v\|\|Rb_w\|}\right)_{v,w\in \widehat V},\quad R:=\nplus\sqrt{\|M_n'(\lambda_0)\|}I_{\C^{d_n}},
\end{align*}
with $\dom R:=U^{-1}C(\widehat V)$, satisfies $S_{L}^{\min}=S_{L_{\min}}$ and $S_L^{\loc}=S_{L_{\min}^*}$.

\item[\rm (b)] We have $n_{\pm}(S_{L}^{\min})=n_{\pm}(L_{\min})$ and there is a bijective correspondence between the self-adjoint extensions of $L_{\min}$ and the self-adjoint extensions of $S_{L}^{\min}$.
\item[\rm (c)] Assume that $\nplus S_{n0}=S_F\geq \gamma$ with $\gamma>0$ and that $\widetilde{M}$ given by \eqref{tildeM} satisfies $\widetilde{M}(\lambda)\rightrightarrows-\infty$ for $\lambda\rightarrow-\infty$. Let $\widetilde{L}$ be a self-adjoint extension of $L_{\min}$ which is semi-bounded from below then $S_{\widetilde{L}}$ is semi-bounded from below.
\end{itemize}
\end{Theorem}

\begin{proof}
For the proof of (a), we use the regularized boundary triplet $\{\Gi,\widetilde{\Gamma}_0,\widetilde{\Gamma}_1\}$ defined in \eqref{gutestripel} for $f=(f_n)_{n\in\N}\in\dom S^*$ as
\begin{align*}
\widetilde{\Gamma}_0f&=(\sqrt{\|M_n'(\lambda_0)\|}\Gamma_0^{(n)}f_n)_{n\in\N},\\
\widetilde{\Gamma}_1f&=(\|M_n'(\lambda_0)\|^{-1/2}(\Gamma_1^{(n)}-M_n(\lambda_0)\Gamma_0^{(n)})f_n)_{n\in\N}.
\end{align*}
Consider $M:=\nplus M_n(\lambda_0)$ with $\dom M:=U^{-1}C(\widehat{V})$ and let $\widetilde{L}_{\min}$ be given by
\[
\dom\widetilde{L}_{\min}:=R U^{-1}C(\widehat{V}),\quad \widetilde{L}_{\min}f:=P_{\overline{\ran R}}R^{-1}(L-M)R^{-1}f.
\]
We show that
\begin{align}
\label{firststep}
\dom S_{L}^{\min}=\{f\in\dom S^* | \widetilde{L}_{\min}\widetilde{\Gamma}_0f=P_{\overline{\ran R}}\widetilde{\Gamma}_1f,\ \widetilde{\Gamma}_0f\in RU^{-1}C(\widehat{V}),\ \supp(P_{\Gi_v}\Gamma_1^v)_{v\in V} \mbox{ finite}\}.
\end{align}
Let $f\in\dom S^*$ be in the set on the right hand side of \eqref{firststep}.
Obviously $\supp\Gamma_0f$ is finite and rewriting the conditions on the right hand side of \eqref{firststep} we obtain
\begin{align*}
P_{\overline{\ran R}}R^{-1}(L-M)R^{-1}R(\Gamma_0^{(n)}f_n)_{n\in\N}&= \widetilde{L}_{\min}\widetilde{\Gamma}_0f\\
&=P_{\overline{\ran R}}\widetilde{\Gamma}_1f\\&=P_{\overline{\ran R}}(\|M_n(\lambda_0)\|^{-1/2}(\Gamma_1^{(n)}-M_n(\lambda_0)\Gamma_0^{(n)})f_n)_{n\in\N}
\end{align*}
and therefore
\begin{align}
\label{gleich}
P_{\overline{\ran R}}R^{-1}L(\Gamma_0^{(n)}f_n)_{n\in\N}=P_{\overline{\ran R}}(\|M_n(\lambda_0)\|^{-1/2}\Gamma_1^{(n)}f_n)_{n\in\N}.
\end{align}
Note that $(\|M_n(\lambda_0)\|^{-1/2}\Gamma_1^{(n)}f_n)_{n\in\N}\in\Gi$, since $\widetilde{\Gamma}_1f\in\Gi$ and $\supp\Gamma_0f$ is finite. The definition of $R$ implies that $\{Rb_w\}_{w\in \widehat{V}}$ is an orthogonal basis of $\ran R$. Furthermore, we have from \eqref{gleich} that for all $w\in\widehat V$
\[
(P_{\overline{\ran R}}R^{-1}L(\Gamma_0^{(n)}f_n)_{n\in\N},Rb_w)=(P_{\overline{\ran R}}(\|M_n(\lambda_0)\|^{-1/2}\Gamma_1^{(n)}f_n)_{n\in\N},Rb_w)
\]
which is equivalent to
\[
((L_v\Gamma_0^vf)_{v\in V},b_w)=(L(\Gamma_0^{(n)}f_n)_{n\in\N},b_w)=((\Gamma_1^{(n)}f_n)_{n\in\N},b_w)=(P_{\Gi_v}(\Gamma_1^{v}f)_{v\in V},b_w)
\]
for all $w\in\widehat V$. Note that $(\Gamma_1^{(n)}f_n)_{n\in\N}$ and $L(\Gamma_0^{(n)}f_n)_{n\in\N}$ are in general not in $\Gi$ but the formal scalar product of these sequences with $b_v$ exists, because the support of $b_v$ is finite.
Since for each $v\in V$ there exists a subset of $\{b_w\}_{w\in \widehat V}$ which is an orthogonal basis for $\Gi_v$, we see that
\[
L_v\Gamma_0^vf=P_{\Gi_v}\Gamma_1^vf
\]
for all $v\in V$ and all $f$ in the set of the right hand side of \eqref{firststep}.
Moreover, $\widetilde{\Gamma}_0f\in RU^{-1} C(\widehat{V})$, hence $\Gamma_0f\in U^{-1} C(\widehat{V})$ and,
by construction, $\Gamma_0^vf\in\Gi_v$ follows.
Thus we have proven that $f\in\dom S_{L}^{\min}$.

Assume conversely that $f\in\dom S_L^{\min}$ then we have that for finitely many $v\in V$ that
\[
L_v\Gamma_0^vf=P_{\Gi_v}\Gamma_1^v,\quad \Gamma_0^vf\in\Gi_v
\]
and $\Gamma_0^vf=P_{\Gi_v}\Gamma_1^vf=0$  otherwise.
Obviously $\widetilde{\Gamma}_0f\in RU^{-1}C(\widehat{V})$ and $\supp(P_{\Gi_v}\Gamma_1^vf)_{v\in V}$
is finite. Furthermore, it is also clear from the calculations in the first part of the proof,
that for all $w\in\widehat V$
\[
(P_{\overline{\ran R}}R^{-1}L(\Gamma_0^{(n)}f_n)_{n\in\N},Rb_w)=(P_{\overline{\ran R}}(\|M_n(\lambda_0)\|^{-1/2}\Gamma_1^{(n)}f_n)_{n\in\N},Rb_w)
\]
holds. Since span$\, \{Rb_w\}_{w\in\widehat V}$ is dense in $\overline{\ran R}$ we have
\[
\widetilde L_{min}\widetilde\Gamma_0f=P_{\overline{\ran R}}\widetilde\Gamma_1f.
\]
Thus the identity \eqref{firststep} holds.

We apply Lemma \ref{unitary} to obtain a different representation of $S_{\widetilde{L}_{\min}}$ in terms of the boundary triplet
$\{\hat U\Gi,\hat U\Gamma_0,\hat U\Gamma_1\}$ where $\hat U:\ran R\rightarrow\ell^2(\widehat V)$ is given by $Rb_w\mapsto \|Rb_w\|e_w$. and with the operator  $L_{\min}=\hat U\widetilde L_{\min}\hat U^*$ which is given by
\begin{align*}
\frac{(\widetilde{L}_{\min}Rb_v,Rb_w)}{\|Rb_v\|\|Rb_w\|}
&=\frac{(P_{\overline{\ran R}}R^{-1}(L-M)R^{-1}Rb_v,Rb_w)}{\|Rb_v\|\|Rb_w\|}\\[1ex]
&=\frac{(R^{-1}(L-M)b_v,P_{\overline{\ran R}}Rb_w)}{\|Rb_v\|\|Rb_w\|}\\[1ex]
&=\frac{((L-\plus M_n(\lambda_0))b_v,b_w)}{\|Rb_v\|\|Rb_w\|}=(L_{\min})_{v,w}
\end{align*}
The assertion (b) follows immediately from (a) and Proposition \ref{bekannt}. An application of Proposition \ref{negEWundso} (b) yields (c).
\end{proof}

Under the assumption that the direct sum triplet \eqref{dstriplet} is a boundary triplet for $S^*$, we have that $\nplus M_n(\lambda_0)$, $R$ and $R^{-1}$ are bounded and we obtain the following special case of Theorem \ref{dieDS}. For quantum graphs with edge length bounded from below, this result was also obtained in \cite{LSV14}.

\begin{Corollary}
Assume that the triplet $\{\Gi,\Gamma_0,\Gamma_1\}$ given by \eqref{dstriplet} is a boundary triplet for $S^*$, then $S_{L}^{\loc}$ has the following properties:
\begin{itemize}
\item[\rm (a)] $S_L^{\loc}$ is self-adjoint.
\item[\rm (b)] Assume that $S^*|_{\ker\Gamma_0}=S_F\geq\gamma$ with $\gamma>0$ and that $\nplus M_n(\lambda)\rightrightarrows-\infty$, as $\lambda\rightarrow-\infty$, then $S_L^{\loc}$ is semi-bounded from below if and only if there exists  $C>-\infty$ with $(L_vx,x)\geq C\|x\|^2$ for all $x\in\Gi_v$ and all $v\in V$.
\end{itemize}
\end{Corollary}
\begin{proof}
Since $S_L^{\loc}$ is closed, it remains to show by Theorem \ref{dieDS} that $L_{\min}$ is essentially self-adjoint.
Every $L_v$ is unitarily equivalent to a diagonal matrix and therefore the operator
$(\frac{(Lb_v,b_w)}{\|Rb_v\|\|Rb_w\|})_{v,w\in\widehat{V}}$ is unitarily equivalent to a densely defined multiplication operator
 on $\ell^2(\widehat V)$, and hence essentially self-adjoint.
 Since  $\{\Gi,\Gamma_0,\Gamma_1\}$ is a boundary triplet, \cite[Theorem 2.12]{CMP13} implies that the operators $R, R^{-1}$ and $\nplus M_n(\lambda_0)$ are bounded. Therefore  $L_{\min}$ is just a bounded and symmetric perturbation of an essentially self-adjoint operator and hence essentially self-adjoint according to the Kato-Rellich theorem \cite[Theorem V.4.4]{K80}. Assertion (b)
is a consequence of the boundedness of $R$, $R^{-1}$ and of $\nplus M_n(\lambda_0)$ and follows from Theorem \ref{dieDS} (a).
\end{proof}

Since \eqref{dstriplet} is in general not a boundary triplet, we use the results of \cite{GHKLW13,KL11} to provide conditions on the self-adjointness of $S_{L}^{\loc}$ and the  discreteness of the spectrum of all self-adjoint extensions in the theorem below. For this we associate with $S_{L}^{\loc}$ the formal discrete Laplacian $D_{L}$ on the weighted space
\[
\ell^2(\widehat{V},m):=\bigg\{(x_v)_{v\in\widehat{V}}\in\C^{\widehat{V}}~\big|~\sum_{v\in \widehat{V}}m(v)|x_v|^2<\infty\bigg\}
\]
with $m(v):=\|Rb_v\|^2$, where $b_v$ is an element of an orthogonal basis of the subspace $\Gi_{\Vc}$ defined in \eqref{defGV} and the scalar product in $\ell^2(\widehat{V},m)$ is given by
$$
(x,y)_m:=\sum_{v\in \widehat{V}}m(v)x_v\overline{y_v}.
$$
We define an operator $D_L$ with domain $\dom D_L:=C(\widehat{V})$ via
\begin{align}
\begin{split}
\label{dislap}
(D_{L}f)_v&:=\frac{1}{\|Rb_v\|^2}\bigg(\sum_{w\in \widehat{V}}b(v,w)(f_v-f_w)+c(v)f_v\bigg),\\
b(v,w)&:=\left(\Big(\nplus M_n(\lambda_0)-L\Big)b_v,b_w\right),\quad v\neq w,\quad b(v,v):=0,\\
c(v)&:=\left(\Big(L-\nplus M_n(\lambda_0)\Big)b_v,b_v\right)-\sum_{w\in \widehat{V}}b(v,w).
\end{split}
\end{align}
The elements of $\{b_v\}_{v\in\widehat{V}}$ have finite support, and if $b_{v_1}$ and $b_{v_2}$ are elements of a basis for $\Gi_{w_1}$ and $\Gi_{w_2}$ with $w_1\neq w_2$ then $\supp b_{v_1}\cap\supp b_{v_2}=\emptyset$. Also, the support of $\Big(\nplus M_n(\lambda_0)-L\Big)b_v$
(considered as a sequence) is finite. Hence, for fixed $w\in \widehat{V}$ we have
 $b(v,w)\neq 0$ for only finitely many $v\in \widehat{V}$.
As in \cite{EKMN17,KL10} we consider the \textit{weighted degree}
\begin{align}
\label{defDeg}
{\rm Deg}:\widehat{V}\rightarrow(0,\infty),\quad v\mapsto\frac{1}{\|Rb_v\|^2}\sum_{w\in\widehat{V}}b(v,w).
\end{align}
\begin{Theorem}
\label{genresultvorne}
Consider the operator $S_{L}^{\loc}$ and the associated discrete Laplacian \eqref{dislap}. Assume that $b(v,w)\geq 0$ holds for all $v,w\in \widehat{V}$. Then the following holds.
\begin{itemize}
\item[\rm (a)] The operator $S_{L}^{\loc}$ is self-adjoint if one of the following conditions holds.
\begin{itemize}
\item[\rm (i)] Assume that $\inf_{v\in \widehat{V}}\frac{c(v)}{\|Rb_v\|^2}>-\infty$ and that for all sequences $\{v_n\}_{n\in\N}$ in $\widehat{V}$ with $b(v_n,v_{n+1})>0$ for all $n\in\N$ we have $\sum_{n=1}^{\infty} \|Rb_{v_n}\|^2=\infty$.
\item[\rm (ii)] The weighted degree ${\rm Deg}$ is bounded.
\end{itemize}
\item[\rm (b)] All self-adjoint extensions of $S_L^{\min}$ are in one-to-one correspondence with the self-adjoint extensions of $D_L$.
\item[\rm (c)] All self-adjoint extensions $\widehat{S}$ of $S_{L}^{\min}$ satisfy  $(\widehat{S}-\lambda)^{-1}\in\mathfrak{S}_1(\Hi)$ for some  $\lambda\in\rho(\widehat S)$ if the following conditions hold.
\begin{itemize}
\item[\rm (i)] For all $v,w\in \widehat{V}$ there exists $k\in\N$ and  $v_0,\ldots,v_k$ such that $v_0=v$, $v_k=w$ and $b(v_i,v_{i+1})>0$ for all $i=0,\ldots,k-1$.
\item[\rm (ii)] Let $\big(\plus S_{n0}-\lambda\big)^{-1}\in\mathfrak{S}_1(\Hi)$ for $\lambda\in\rho\big(\plus S_{n0}\big)$.
\item[\rm (iii)] Let $\sum\limits_{v,w\in V, b(v,w)\neq 0}b(v,w)^{-1}<\infty$, $\sum_{v\in \widehat{V}}\|R b_v\|^2<\infty$, $\inf_{v\in \widehat{V}}\frac{c(v)}{\|Rb_v\|^2}>-\infty$.
\end{itemize}
\item[\rm (d)] Assume that $\plus S_{n0}=S_F\geq\gamma$ with $\gamma>0$ and that $\widetilde{M}(\lambda)\rightrightarrows-\infty$ for $\lambda\rightarrow-\infty$ and $\inf_{v\in \widehat{V}}\frac{c(v)}{\|Rb_v\|^2}>-\infty$ then all
    self-adjoint extensions of  $S_L^{\min}$ are semi-bounded from below.
\end{itemize}
\end{Theorem}
\begin{proof}
First, we prove the results for $c(v)\geq 0$. To prove (a), we use that by Proposition \ref{adjoint} the operator $S_{L}^{\loc}$ is
closed. It remains by Theorem \ref{dieDS}~(b) and Proposition \ref{bekannt}~(a) to show that the operator given by
\[
(L_{\min})_{v,w}:=\frac{((L-\plus M_n(\lambda_0))b_v,b_w)}{\|Rb_v\|\|R b_w\|}
\]
is essentially self-adjoint on $C(\widehat{V})$. A straight forward calculation shows that $L_{\min}$ is unitary equivalent via  $U:\ell^2(\widehat{V},m)\rightarrow\ell^2(\widehat{V})$, $(x_v)_{v\in \widehat{V}}\mapsto (\|Rb_v\|x_v)_{v\in \widehat{V}}$ to the operator $D_{L}$.
The assumption in (i) on the sequences $(v_n)_{n\in\N}$ in $V$ and the invariance $D_{L}C(\widehat{V})\subseteq C(\widehat{V})$  allows us to apply \cite[Theorem  6]{KL11} which yields the essential self-adjointness of $D_{L}$ on $C(\widehat{V})$. This shows the essential self-adjointness of $S_{L}^{\min}=S_{L_{\min}}$ by Proposition \ref{bekannt} (a).
The assumption (ii) implies by \cite[Theorem 11]{KL10} that $D_0$ given by $D_L$ with $c(v)=0$ for all $v\in\widehat{V}$ is bounded. Therefore $D_L$ on $C(\widehat{V})$ is the bounded and symmetric perturbation of the essentially self-adjoint multiplication operator $(x_v)_{v\in \widehat{V}}\mapsto(\frac{c(v)}{\|Rb_v\|^2}x_v)_{v\in \widehat{V}}$ on $C(\widehat{V})$ hence essentially self-adjoint because of the Kato-Rellich theorem \cite[Theorem V.4.4]{K80}.
The correspondence in (b) is a consequence of Theorem \ref{dieDS} (b).

The assertion (c) follows from \cite[Theorem 5.1]{GHKLW13} applied to $D_{L}$ which shows that all self-adjoint extensions of $D_{L}$  have resolvents in $\mathfrak{S}_1(\ell^2(\widehat V,m))$. Note that the assumptions of this Theorem 5.1 are satisfied because of $\sum_{v\in \widehat{V}}m(v)=\sum_{v\in \widehat{V}}\|Rb_v\|^2<\infty$ and (i) and (iii), see also \cite[Example 4.6]{GHKLW13}. The assumption (ii) that $\big(\plus S_{n0}-\lambda\big)^{-1}\in\mathfrak{S}_1(\Hi)$ for $\lambda\in\rho\big(\plus S_{n0}\big)$ together with Proposition \ref{bekannt}~(e) imply that $(\widehat{S}-\lambda)^{-1}\in\mathfrak{S}_1(\Hi)$. This proves (c).

Let $\widehat{S}_{L}$ be an extension of $S_{L}^{\min}$ and $\widehat{D}_{L}$ be an extension of $D_{L}$ on $C(\widehat{V})$ with  $\widehat{S}_{L}=S_{\widehat{D}_{L}}$.
It was shown in \cite[p.\ 206]{KL11} that $\widehat{D}_L$ has the same action as $D_L$. For $f\in \dom \widehat{D}_{L}$
with $(f,f)_m=1$ we see from $b(v,w)\geq 0$ that
\begin{align*}
(\widehat{D}_{L} f,f)_m&=\sum_{v\in\widehat  V}m(v)(\widehat{D}_Lf)_v\overline{f_v}\\&=\frac{1}{2}\sum_{v,w\in\widehat V}b(v,w)|f_v-f_w|^2+\sum_{v\in\widehat V}c(v)|f_v|^2\\
&\geq\sum_{v\in\widehat V}c(v)|f_v|^2
\geq \inf_{v\in\widehat V}\frac{c(v)}{\|Rb_v^2\|}(f,f)_m
=\inf_{v\in \widehat{V}}\frac{c(v)}{\|Rb_v\|^2}.
\end{align*}
Proposition \ref{negEWundso} (a) applied to the regularized boundary triplet  $\{\Gi,\widetilde{\Gamma}_0,\widetilde{\Gamma}_1\}$ from Theorem \ref{handyreg} yields that $S_{\widehat{D}_{L}}$ is semi-bounded from below. Here we used that due to \eqref{invariant} we have $S^*|_{\ker\widetilde{\Gamma}_0}=S_F$.

Assume now that $\inf_{v\in \widehat{V}}\frac{c(v)}{\|Rb_v\|^2} >-\infty$ holds. Then the operator $\widehat{D}_{L}$ is the bounded perturbation of an operator $\widehat{D}_{L}^+$ where we replace $c(v)$ with its positive part $c(v)^+:=\max\{c(v),0\}$. Therefore we can apply the previous arguments to $\widehat{D}_{L}^+$. By assumption, $\widehat{D}_{L}$ is a bounded perturbation of $\widehat{D}_{L}^+$ again the Kato-Rellich theorem shows that self-adjointness is preserved which proves (a) and (c). Furthermore, (d) follows from Proposition \ref{negEWundso} (b).
\end{proof}

\section{Gesztesy-\v{S}eba realizations of Dirac operators on metric graphs}
\label{sec:Dirac}
In this section, we define the Gestezy-\v{S}eba realization of Dirac operators on a locally finite graphs given by a set of vertices $V$ and a set of edges $E$. On each edge $e\in E$ with finte length $\ell(e)$ we consider the Dirac operator
\[
D_e:=\begin{pmatrix}c^2/2 & -ic\frac{d}{dx_e}\\ -ic\frac{d}{dx_e}&-c^2/2\end{pmatrix},\ \dom D_e:=H^{1}_0(0,\ell(e))\otimes\C^2,
\]
where $c$ denotes the speed of light. It was shown in \cite[Lemma 3.1]{CMP13} that a boundary triplet for $D_e^*$ is given by
\begin{align*}
\Gi_e:=\C^2,\quad \hat \Gamma_0^{(e)}\begin{pmatrix}\psi_{e,1}\\ \psi_{e,2}\end{pmatrix}:=\begin{pmatrix}\psi_{e,1}(0+)\\ ic\psi_{e,2}(\ell(e)-)\end{pmatrix},\quad \hat \Gamma_1^{(e)}\begin{pmatrix}\psi_{e,1}\\ \psi_{e,2}\end{pmatrix}:=\begin{pmatrix}ic\psi_{e,2}(0+)\\ \psi_{e,1}(\ell(e)-)\end{pmatrix}
\end{align*}
with the Weyl function for $\lambda\in\rho(D_{e}^*|_{\ker\hat \Gamma_0^{(e)}})$
\[
\hat M_e(\lambda):=\frac{1}{\cos(\ell(e)k(\lambda))}\begin{pmatrix}ck_1(\lambda)\sin(\ell(e)k(\lambda))& 1 \\ 1& (ck_1(\lambda))^{-1}\sin(\ell(e)k(\lambda))\end{pmatrix},
\]
where we abbreviate
\[
k(\lambda):=c^{-1}\sqrt{\lambda^2-(c^2/2)^2},\quad k_1(\lambda):=\frac{ck(\lambda)}{\lambda+c^2/2}=\sqrt{\frac{\lambda-c^2/2}{\lambda+c^2/2}}
\]
with $\sqrt{\cdot}$ such that $k(x)>0$ for $x>\frac{c^2}{2}$. Under the assumption that $\sup_{e\in E}\ell(e)<\infty$, it was shown in \cite[Equation (3.56)]{CMP13} that for some $\varepsilon>0$ we have $(\frac{c^2}{2}-\varepsilon,\frac{c^2}{2}+\varepsilon)\subseteq\bigcap_{e\in E}\rho(D_{e}^*|_{\ker\hat \Gamma_0^{(e)}})$ and
\begin{align}
\label{wertebeic2}
\hat M_e\left(\frac{c^2}{2}\right)=\begin{pmatrix}0 & 1 \\ 1 & \ell(e)\end{pmatrix},\quad \hat M_e'\left(\frac{c^2}{ 2}\right)=\begin{pmatrix}\ell(e) & \frac{\ell(e)^2}{2} \\ \frac{\ell(e)^2}{2} & \frac{\ell(e)}{c^2}+\frac{\ell(e)^3}{3}\end{pmatrix}.
\end{align}

To describe a point interaction on a graph, we consider the boundary triplet for $D_e^*$ given by a unitary transformation
\[
\begin{pmatrix}
\Gamma_0^{(e)}\\ \Gamma_1^{(e)}
\end{pmatrix}:=\begin{bmatrix}W_{00} & W_{01}\\ W_{10} & W_{11}\end{bmatrix}\begin{pmatrix}
\hat \Gamma_0^{(e)}\\ \hat \Gamma_1^{(e)}
\end{pmatrix}=\begin{pmatrix}
1 & 0 & 0 & 0\\ 0 & 0 & 0 & i\\ 0 & 0 & 1 & 0\\ 0 & -i & 0 & 0\\
\end{pmatrix} \begin{pmatrix} \psi_{e,1}(0+)\\[0.75ex] ic\psi_{e,2}(\ell(e)-) \\[0.75ex]
ic\psi_{e,2}(0+)\\[0.75ex] \psi_{e,1}(\ell(e)-)
\end{pmatrix}
\]
with $W_{00},W_{01},W_{10},W_{11}\in\C^{2\times 2}$ and therefore
\[
\Gamma_0^{(e)}\begin{pmatrix}
\psi_{e,1}\\ \psi_{e,2}
\end{pmatrix}=\begin{pmatrix} \psi_{e,1}(0+)\\ i\psi_{e,1}(\ell(e)-)\end{pmatrix},\quad \Gamma_1^{(e)}\begin{pmatrix}
\psi_{e,1}\\ \psi_{e,2}
\end{pmatrix}:=\begin{pmatrix}ic\psi_{e,2}(0+)\\ c\psi_{e,2}(\ell(e)-)
\end{pmatrix}.
\]
It was shown in \cite{DHMS00} that such a unitary transformation leads to a boundary triplet with the Weyl function given by
\begin{align}
\label{transform}
\begin{split}
M_e(\lambda)&=(W_{10}+W_{11}\hat M_e(\lambda))(W_{00}+W_{01}\hat M_e(\lambda))^{-1}\\[1ex]
&=\frac{ck_1(\lambda)}{\sin(\ell(e)k(\lambda))}\begin{pmatrix}\cos(\ell(e)k(\lambda))&-i\\i &-\cos(\ell(e)k(\lambda))\end{pmatrix}
\end{split}
\end{align}
for all $\lambda\in\rho(D_{e}^*|_{\ker\hat \Gamma_0^{(e)}})\cap \rho(D_{e}^*|_{\ker \Gamma_0^{(e)}})$.

Introduce the set $I_v$ with $(e,0)\in I_v$ if $e\in E$ and $e$ has $v$ as initial vertex and $(e,1)\in I_v$ if $e\in E$ and $e$ has $v$ as terminal vertex. The vectors $b_v\in\Gi$ are given by
\[
(b_v)_{(e,t)}:=\begin{cases}1, &\text{if $(e,0)\in I_v$,}\\ i, &\text{if $(e,1)\in I_v$,}\\ 0, & \text{if $(e,t)\notin I_v$.}\end{cases}
\]
Let $(\alpha(v))_{v\in V}$ be a real sequence.
The operator ${\rm GS}_{\alpha}$ is given by
\begin{align*}
\dom {\rm GS}_{\alpha}:=\left\{(\psi_1,\psi_2)^{\top}\in \Eplus D_e^*~:~\psi_1\in \mathcal{C}(G),\ ic\sum_{(e,t)\in I_v}\sgn(e,t)\psi_{e,2}(t\ell(e))=\alpha(v)\psi_1(v),\ v\in V\right\},
\end{align*}
where $\mathcal{C}(G)$ is the set of continuous functions on $G$ viewed as a metric space and $\psi_1(v)$ is the value of $\psi_1$ at the vertex $v$. We follow here \cite{CMP13}
and call this operator \emph{Gestesy-\v{S}eba realization}.

If $\sup_{e\in E}\ell(e)<\infty$, it can easily be seen from \eqref{transform} that for some $\varepsilon>0$ we have $(\frac{c^2}{2}-\varepsilon,\frac{c^2}{2}+\varepsilon)\subseteq\bigcap_{e\in E}\rho(D_{e}^*|{\ker\Gamma_0^{(e)}})$ and that
\begin{align*}
M_e\left(\frac{c^2}{2}\right)&=\frac{1}{\ell(e)}\begin{pmatrix}1 & -i\\ i & 1\end{pmatrix}.
\end{align*}
We also see from \eqref{transform} and \eqref{wertebeic2} with $T:=\left(W_{00}+W_{01}\hat M_e\left(\frac{c^2}{2}\right)\right)^{-1}$ that
\begin{align*}
M_e'\left(\frac{c^2}{2}\right)
&=W_{11}\hat M_e'\left(\frac{c^2}{2}\right)T-\left(W_{10}+W_{11}\hat
M_e\left(\frac{c^2}{2}\right)\right)TW_{01} \hat M'_e\left(\frac{c^2}{2}\right)T\\[1ex]
&=\begin{pmatrix} \frac{1}{\ell(e)c^2} + \frac{\ell(e)}{3}& -\frac{ i\ell(e)}{2} + \frac{i}{\ell(e)c^2} + \frac{i\ell(e)}{3}\\
\frac{i\ell(e)}{2} - \frac{i}{\ell(e)c^2} - \frac{i\ell(e)}{3} & \frac{1}{\ell(e)c^2} + \frac{\ell(e)}{3}\end{pmatrix}
\end{align*}
and this implies
\begin{align}
\label{lowerbound}
\left\|M_e'\left(\frac{c^2}{2}\right)\right\|\geq (1,0)M_e'\left(\frac{c^2}{ 2}\right)\begin{pmatrix}1\\ 0\end{pmatrix}=\frac{1}{\ell(e)c^2} +\frac{\ell(e)}{3}\geq\frac{1}{\ell(e)c^2}.
\end{align}
Furthermore, we define
$$
\Gi_v:=\Span\{1_v\},\quad 1_v:=((b_{v})_{(e,t)})_{(e,t)\in I_v}\quad  \mbox{and} \quad L_v1_v:=\frac{\alpha(v)}{\deg v}1_v.
$$
We have according to \eqref{dislap} for $v\neq w$
\begin{align*}
b(v,w)&:=\left(\Big(\Eplus M_e\left(\frac{c^2}{2}\right)-L\Big)b_v,b_w\right)=
\left(\Eplus M_e\left(\frac{c^2}{2}\right)b_v,b_w\right)\\ &=\begin{cases}\ell(e)^{-1} & \text{if $e=vw\in E$,}\\[1ex]
0 & \text{if $e=vw\notin E$,}\end{cases}\\
\end{align*}
and we see for $v\in V$
\begin{align*}
c(v)&:=\left(\Big(L-\Eplus M_e\left(\frac{c^2}{2}\right)\Big)b_v,b_v\right)-\sum_{w\in V, w\neq v}\left(\Eplus M_e\left(\frac{c^2}{2}\right)b_v,b_w\right)\\[1ex]
 &=(L_vb_v,b_v)=\alpha(v).
\end{align*}

As an application of Theorem \ref{genresultvorne}, we have the following result on the self-adjointness of the Gesztesy-\v{S}eba realizations.
\begin{Proposition}
Consider a locally finite graph with set of vertices $V$ and set of edges $E$ and let $\{\alpha(v)\}_{v\in V}$ be a real-valued sequence. Then the operator ${\rm GS}_\alpha$ is a locally finite extension of $\Eplus D_e$ and if $\sup_{e\in E}\ell(e)<\infty$ then ${\rm GS}_{\alpha}$ is self-adjoint.
\end{Proposition}
\begin{proof}
We show that ${\rm GS}_{\alpha}$ is a locally finite extension of $\Eplus D_e$.
Let $I:=E\times \{0,1\}$, then
\begin{align*}
\Gamma_0\psi&=(\Gamma_0^{(e,t)}(\psi_{e,1},\psi_{e,2})_{e\in E})_{(e,t)\in I}=(i^t\psi_{e,1}(t\ell(e)))_{(e,t)\in I},\\
\Gamma_1\psi&=(\Gamma_1^{(e,t)}(\psi_{e,1},\psi_{e,2})_{e\in E})_{(e,t)\in I}=(ci^{1-t}\psi_{e,1}(t\ell(e)))_{(e,t)\in I}
\end{align*}
and therefore
\[
\Gamma_0^v(\psi_{e,1},\psi_{e,2})_{e\in E}:=(i^t\psi_{e,1}(t\ell(e)))_{(e,t)\in I_v},\quad \Gamma_1^v(\psi_{e,1},\psi_{e,2})_{e\in E}:=(ci^{1-t}\psi_{e,2}(t\ell(e)))_{(e,t)\in I_v}.\]
Since $\Gi_v=\Span\{1_v\}$ for all $v\in V$, we see that $\psi_1\in\mathcal{C}(G)$ is equivalent to the condition $
\Gamma_0^v(\psi_{e,1},\psi_{e,2})_{e\in E}\in\Gi_v$ for all $v\in V$. Moreover, it is easy to see that the sum condition in the definition of $\dom{\rm GS}_{\alpha}$ is equivalent to
\begin{align*}
    P_{\Gi_v}\Gamma_1^v(\psi_{e,1},\psi_{e,2})_{e\in E}&=\frac{1}{\|1_v\|^2}(\Gamma_1^v(\psi_{e,1},\psi_{e,2})_{e\in E},1_v)1_v\\
    &=\frac{1}{\deg v}\sum_{(e,t)\in I_v}\overline{(b_v)_{(e,t)}}ci^{1-t}\psi_{e,2}(t\ell(e))1_v\\
    &=\frac{ic}{\deg v}\sum_{(e,t)\in I_v}\sgn(e,t)\psi_{e,2}(t\ell(e))1_v\\&=\frac{\alpha(v)}{\deg v}\psi_1(v)1_v=L_v\Gamma_0^v(\psi_{e,1},\psi_{e,2})_{e\in E}.
\end{align*}
Thus, we have seen that ${\rm GS}_{\alpha}$ is a locally finite extension of $\Eplus D_e$.

For $\sup_{e\in E}\ell(e)<\infty$, the assumptions of Theorem \ref{dieDS} are fulfilled.
To see that ${\rm GS}_{\alpha}$ is self-adjoint, we apply Theorem \ref{genresultvorne} (a). The estimate \eqref{lowerbound} implies  that the weighted degree \eqref{defDeg} satisfies
\begin{align*}
{\rm Deg}(v)=\frac{\sum_{w\in V}b(v,w)}{\|Rb_v\|^2}=\frac{\sum_{w\in V}b(v,w)}{\sum_{e=vw}\left\|M_e'(\frac{c^2}{2})\right\|}&\leq \frac{\sum_{e=vw}\ell(e)^{-1}}{\sum_{e=vw}\frac{1}{c^2\ell(e)}}=c^2<\infty
\end{align*}
for all $v\in V$, where the summation $\sum_{e=vw}$ is taken over all edges $e$ that contain $v$ as a vertex. Hence, according to Theorem \ref{genresultvorne}, $GS_{\alpha}$ is self-adjoint.
\end{proof}


\begin{thebibliography}{99}

\bibitem{A61}
R.\ Arens, \textit{Operational calculus of linear relations}, Pacific J. Math \textbf{11} (1961), 9--23.

\bibitem{BL10}
J.\ Behrndt, A.\ Luger, \textit{On the number of negative eigenvalues of the Laplacian on a metric graph},\ J.\ Phys A: Math. Theor.\ \textbf{43} (2010), 474006.

\bibitem{BK13}
G.\ Berkolaiko, P.\ Kuchment: \textit{Introduction to Quantum Graphs}, AMS, Providence, RI, 2013.

\bibitem{BEH08}
J.\ Blank, P.\ Exner, M.\ Havli\v{c}ek: \textit{Hilbert Space Operators in Quantum Physics, 2nd edition}, Springer, New York, 2008.

\bibitem{CMP13}
R.\ Carlone, M.\ Malamud, A.\ Posilicano, \textit{On the spectral theory of Gesztesy-\v{S}eba realizations of 1-D Dirac operators with point interactions on a discrete set},\ J.\ Differential \ Equations \textbf{254} (2013), 3835--3902.




\bibitem{DHMS00}
V.\ Derkach, S.\ Hassi, M.\ Malamud, H.\ de Snoo, \textit{Generalized
resolvents of symmetric operators and admissibility},\ \textit{Methods Funct.\ Anal.\ Topology} \textbf{6} (2000),  24--55.

\bibitem{DHMS06}
V.\ Derkach, S.\ Hassi, M.\ Malamud, H.\ de Snoo, \textit{Boundary relations and their Weyl families}, Trans. Amer. Math. Soc. \textbf{358} (2006), 5351--5400.


\bibitem{DM91}
V.\ Derkach, M.\ Malamud, \textit{Generalised resolvents and the boundary value problems for Hermitian operators with gaps}, \ J.\ Func.\ Anal.
\textbf{95} (1991), 1--95.

\bibitem{DM95}
V.\ Derkach, M.\ Malamud, \textit{The extension theory of {H}ermitian operators and the moment problem}, \ J.\ Math.\ Sci. \textbf{73} (1995), 141--242.



\bibitem{EKMN17}
P.\ Exner, A.\ Kostenko, M.\ Malamud, H.\ Neidhardt, \textit{Spectral theory of infinite quantum graphs}, \ arXiv:1705.01831v1, 1--43.

\bibitem{GHKLW13}
A.\ Georgakopoulos, S.\ Haeseler, M.\ Keller, D.\ Lenz, R.\ Wojciechowski, \textit{Graphs of finite measure}, J.\ Math.\ Pure.\ Appl.\ \textbf{103} (2015), 1093--1131.


\bibitem{GS87}
F.\ Gesztesy, P.\ \v{S}eba, \textit{New analytically solvable models of relativistic point interactions},\ Lett.\ Math.\ Phys.\
\textbf{13} (1987), 345--358.


\bibitem{GG91}
V.I.\ Gorbachuk, M.L.\ Gorbachuk: \textit{Boundary Value Problems for Operator Differential Equations}, Kluwer Academic Publishers
Group, Dordrecht, 1991.




\bibitem{Jeribi} A.\ Jeribi:  \textit{Spectral Theory
and Applications of Linear Operators and Block Operator Matrices},
Springer-Verlag, New-York, 2015.


\bibitem{K80}
T.\ Kato: \textit{Perturbation Theory for Linear Operators}, 2nd edition,  Springer, Berlin, 1980.



\bibitem{KL10}
M.\ Keller, D.\ Lenz, \textit{Unbounded Laplacians on graphs: basic sepctral properties and the heat equation},\ Math.\ Model.\ Nat.\ Phenom.\ \textbf{5} (2010), 198--224.

\bibitem{KL11}
M.\ Keller, D.\ Lenz, \textit{Dirichlet forms and stochastic completeness of graphs and subgraphs},\ J.\ reine Angew. Math.\ \textbf{666} (2012), 189--223.

\bibitem{K75}
A.\ Kochubei, \textit{Extensions of symmetric operators and binary relations},\ Math.\ Notes
\textbf{17} (1975), 25--28.



\bibitem{KM10}
A.\ Kostenko, M.\ Malamud, \textit{1-D Schr\"{o}dinger operators with local point interactions on a discrete set},\ J.\ Differential \ Equations
\textbf{249} (2010), 259--304.

\bibitem{KS99}
V.\ Kostrykin, R.\ Schrader, \textit{Kirchhoff's rule for quantum wires}, J. Phys.
A \textbf{32}(1999), 595--630.

\bibitem{KS06}
V.\ Kostrykin, R.\ Schrader, \textit{Laplacians on metric graphs: eigenvalues,
resolvents and semigroups}, in: Quantum Graphs and their applications, pages 201-225, AMS, Providence, RI, 2006.



\bibitem{LSV14}
D.\ Lenz, C.\ Schubert, I. Veseli\'{c}, \textit{Unbounded quantum graphs with unbounded boundary conditions}, Mathematische Nachrichten \textbf{287}(2014), 962--979.

\bibitem{MN09}
M.\ Malamud, H.\ Neidhardt, \textit{On the unitary equivalence of absolutely continuous parts of self-adjoint extensions}, Preprint 2009, arXiv:0907.0650.


\bibitem{MN12}
M.\ Malamud, H.\ Neidhardt, \textit{Sturm-Liouville boundary value problems with operator potentials and unitary equivalence}, J.\ Differential \ Equations
\textbf{252} (2012), 5875--5922.

\bibitem{P06}
K.\ Pankrashkin, \textit{Spectra of Schr\"{o}dinger operators on equilateral quantum graphs},\ Lett.\ Math.\ Phys. \textbf{77} (2006), 139--154.

\bibitem{P14}
A.\ Posilicano, \textit{Direct sum of trace maps and self-adjoint extensions},  Arabian Journal of Mathematics \textbf{3} (2014), 437--447.


\bibitem{P08}
O.\ Post, \textit{Equilateral quantum graphs and boundary triples}, in: Analysis on graphs and its applications, pages
469--490. Amer. Math. Soc., Providence, RI, 2008.



\bibitem{S12}
K.\ Schm\"{u}dgen: \textit{Unbounded Self-adjoint Operators on Hilbert Space}, Springer, Dordrecht, 2012.

\bibitem{SSVW15}
C.\ Schubert, C.\ Seifert, J.\ Voigt, M.\ Waurick, \textit{Boundary systems and (skew-)self-adjoint operators on infinite metric graphs},  Math. Nachr. \textbf{288} (2015), 1776--1785.


\bibitem{Tretter2008}
C.\ Tretter:   \textit{Spectral Theory of  Block Operator Matrices
and Applications}, Imperial College Press, London, 2008.

\end{thebibliography}
\end{document}